\documentclass[final,onefignum,onetabnum,onealgnum]{siamart251216}

\usepackage{amsmath,amsfonts,amsopn,amssymb}
\usepackage{graphicx}
\usepackage[caption=false]{subfig}
\graphicspath{ {./Figures/} }
\usepackage{multirow,relsize,makecell}
\usepackage{url}
\ifpdf
\DeclareGraphicsExtensions{.eps,.pdf,.png,.jpg}
\else
\DeclareGraphicsExtensions{.eps}
\fi
\usepackage{algorithm, algorithmic}
\usepackage[shortlabels]{enumitem}

\numberwithin{theorem}{section}
\newsiamremark{assumption}{Assumption}
\newsiamremark{remark}{Remark}
\newsiamremark{question}{Question}
\newsiamremark{example}{Example}
\crefname{assumption}{Assumption}{Assumptions}
\crefname{remark}{Remark}{Remarks}
\crefname{example}{Example}{Examples}

\headers{High-order augmented Lagrangian}{Young-Ju Lee and Jongho Park}

\title{High-order augmented Lagrangian methods\\ and robust multilevel inner solvers\thanks{Funding: Young-Ju Lee's work was supported by NSF-DMS 2208499. Jongho Park's work was supported by the University Development Fund from The Chinese University of Hong Kong, Shenzhen (UDF01004630).}}

\author{
Young-Ju Lee\thanks{Department of Mathematics, Texas State University, San Marcos, TX 78666, USA
  (\email{yjlee@txstate.edu}).}
\and
Jongho Park\thanks{School of Science and Engineering, The Chinese University of Hong Kong, Shenzhen, Guangdong 518172, China
 (\email{jonghopark@cuhk.edu.cn}).}
}

\ifpdf
\hypersetup{
  pdftitle={High-order augmented Lagrangian methods and robust multilevel inner solvers},
  pdfauthor={Young-Ju Lee and Jongho Park}
}
\fi

\usepackage[normalem]{ulem}
\usepackage{color}

\begin{document}

\maketitle
\begin{abstract}
We analyze high-order augmented Lagrangian methods for convex optimization problems with linear constraints and construct robust multilevel inner solvers.
For the high-order proximal point method, we derive energy error estimates under uniform convexity that distinguish the sublinear, linear, and superlinear regimes by comparing the regularization power with the uniform convexity exponent.
Through duality, these estimates yield convergence results for the high-order augmented Lagrangian method.
The primal subproblems become nearly semicoercive as the penalty parameter decreases.
We address this difficulty using multilevel subspace correction methods whose subspaces decompose the constraint kernel.
Applications to nonlinear diffusion, Darcy--Forchheimer flow, and the finite neuron method show how the multilevel design accommodates mean, divergence, and boundary constraints.
Numerical experiments illustrate the outer convergence regimes and demonstrate parameter-robust behavior of the inner solvers across different meshes.
\end{abstract}

% Keywords
\begin{keywords}
high-order augmented Lagrangian method, high-order proximal point method, multilevel method, nearly semicoercive problem
\end{keywords}

% AMS classification
\begin{AMS}
90C25,  % Convex programming
65N55, % Multigrid methods; domain decomposition for boundary value problems
90C46  % Optimality conditions and duality in mathematical programming
\end{AMS}

% Section: Introduction
\section{Introduction}
\label{Sec:Introduction}
% Model problem and TLDR
This paper analyzes high-order augmented Lagrangian methods for convex optimization problems with linear constraints and develops robust multilevel inner solvers.
Let $V$ and $W$ be finite-dimensional vector spaces.
As a model problem, we consider the following convex optimization problem with a linear constraint:
\begin{equation}
\label{model_ALM}
\min_{v \in V} F(v)
\quad \text{subject to} \quad Bv = g,
\end{equation}
where $B \colon V \rightarrow W$ is a linear operator, $F \colon V \rightarrow \mathbb{R}$ is a convex functional, and $g \in W$.
Let $u \in V$ denote a solution of~\eqref{model_ALM}.

% Augmented Lagrangian method
The classical augmented Lagrangian method for solving~\eqref{model_ALM} was first proposed in~\cite{Hestenes:1969,Powell:1969} and is given as follows:
\begin{equation}
\label{ALM_intro}
\begin{aligned}
    u^{(n+1)} &= \operatornamewithlimits{\arg\min}_{v \in V} \left\{ F(v) + ( \lambda^{(n)}, Bv - g) + \frac{1}{2 \epsilon} \| Bv - g \|_{\ast}^2 \right\}, \\
    \lambda^{(n+1)} &= \lambda^{(n)} + \epsilon^{-1} (B u^{(n+1)} - g), 
\end{aligned}
\quad n \geq 0,
\end{equation}
where $\epsilon$ is a positive penalty parameter and $\| \cdot \|_{\ast}$ is induced by an inner product on $W$.
The multipliers in $W^*$ are represented in $W$ using this inner product.
The augmented Lagrangian method~\eqref{ALM_intro} is one of the most successful and widely used approaches for solving constrained optimization problems.
Despite its fundamental nature and its inclusion in standard textbooks on mathematical optimization (see, e.g.,~\cite{Bertsekas:1999,BBF:2013}), active research on the augmented Lagrangian method continues, aiming to broaden its applicability to a wider class of problems and to deepen the underlying theoretical understanding.

For example, in~\cite{KS:2019,KSW:2018}, extensions of the augmented Lagrangian method to Banach space settings and applications to quasi-variational inequalities were studied.
In~\cite{Rockafellar:2023}, a very general framework for the augmented Lagrangian method was developed, covering both convex and nonconvex problems as well as equality and inequality constraints.
From a theoretical perspective, a key observation made in~\cite{Rockafellar:1976b} is the duality between the augmented Lagrangian method and the proximal point method~\cite{Rockafellar:1976a}, which provides valuable insight into the structure of the augmented Lagrangian method and leads to elegant analyses.
A modern presentation of this duality can be found in~\cite[Appendix~A]{PH:2025}, where the relationship is further extended to connect with gradient descent methods in~\cite{Park:DD29}.
Related duality relationships among various convex optimization algorithms are discussed in~\cite{JPX:2025}.

In this paper, we first analyze the convergence rates of the high-order proximal point method~\cite{AN:2024,Nesterov:2021,Nesterov:2023} when the objective functional is uniformly convex~\cite{Park:2022b}, a condition satisfied by many variational problems arising from partial differential equations~(PDEs).
We derive energy error estimates that identify superlinear, linear, and sublinear convergence when the regularization power is greater than, equal to, or less than the uniform convexity exponent, respectively.
Using the duality between proximal point and augmented Lagrangian methods~\cite{JPX:2025,PH:2025,Rockafellar:1976b}, we apply these estimates to analyze high-order augmented Lagrangian methods for~\eqref{model_ALM}.
This yields convergence estimates for the dual energy and corresponding bounds for the primal error.

The efficient solution of the primal subproblems is a central computational issue in high-order augmented Lagrangian methods.
In these subproblems, the penalty term becomes dominant as $\epsilon$ decreases but is constant along the kernel of $B$, so the problems become nearly semicoercive.
This structure can make inner solvers increasingly sensitive to $\epsilon$.
We address this difficulty using multilevel subspace correction methods~\cite{TX:2002,Xu:1992} whose correction spaces decompose the constraint kernel.
In~\cite{LP:2025b}, parallel subspace correction methods for nearly semicoercive convex problems were analyzed as an extension of the theory for nearly singular linear systems~\cite{LWXZ:2007,WZ:2014}.
This theory combines the kernel decomposition with quantitative stability estimates to obtain parameter-robust convergence.
Motivated by this theory, we use successive multilevel corrections and the full approximation scheme (FAS) to formulate nonlinear coarse problems using coarse-level data.

Applications to nonlinear diffusion, Darcy--Forchheimer flow, and the finite neuron method illustrate the versatility of these multilevel solvers.
The corresponding constraint kernels consist of functions with zero mean, divergence-free fields, and functions with homogeneous boundary values, respectively.
We show how to incorporate each kernel into the multilevel correction spaces and compare the resulting solvers with those whose decompositions do not satisfy the kernel decomposition condition.
Numerical experiments demonstrate parameter-robust behavior across three mesh levels and illustrate the three convergence regimes of the outer iteration.

During the preparation of this manuscript, we became aware of the recently published work in~\cite{OBLP:2026}, which is closely related to the outer convergence analysis presented here.
This work analyzes power proximal point and augmented Lagrangian methods, establishing global convergence under implementable inexactness criteria and local superlinear convergence under growth assumptions.
Its numerical treatment of the inner subproblems uses first-order and quasi-Newton methods for H\"older smooth minimization.
Part of our outer convergence analysis overlaps with these results.
Our emphasis is on multilevel subspace correction for the nearly semicoercive primal subproblems, with correction spaces tailored to the constraint kernel and numerical tests of mesh and penalty parameter dependence.

% Paper organization
The rest of this paper is organized as follows.
In \cref{Sec:Preliminaries}, we introduce the notation and present preliminary material.
In \cref{Sec:PPM}, we analyze the convergence rates of the high-order proximal point method under certain uniform convexity assumptions on the energy functional.
In \cref{Sec:ALM}, we recall the high-order augmented Lagrangian method and apply the proximal point estimates to its outer iteration.
In \cref{Sec:Computation}, we discuss computational aspects of the high-order augmented Lagrangian methods.
In \cref{Sec:Applications}, we present several applications of the high-order augmented Lagrangian methods.
In \cref{Sec:Monotone}, we give a brief extension to uniformly monotone inclusions.
Finally, in \cref{Sec:Conclusion}, we conclude the paper with remarks.

% Section: Preliminaries
\section{Preliminaries}
\label{Sec:Preliminaries}
Throughout this paper, let $V$ and $W$ be finite-dimensional vector spaces with dual spaces $V^*$ and $W^*$, respectively.
We identify each space with its bidual and use $(\cdot,\cdot)$ for both inner products and duality pairings.
We equip $V$ and $W^*$ with a norm $\| \cdot \|$, which need not be induced by an inner product.
The dual norm $\| \cdot \|_{\ast}$ is defined by
\begin{equation*}
\| \xi \|_{\ast} = \sup_{\| v \| \leq 1} (\xi,v),
\quad \xi \in V^*,
\end{equation*}
and analogously on $W$.
For a linear operator $B\colon V\to W$, its adjoint $B^t\colon W^*\to V^*$ is defined by
\begin{equation*}
(B^t\sigma,v)=(\sigma,Bv),
\quad v\in V,\ \sigma\in W^*.
\end{equation*}
For a differentiable functional $F\colon V\to\mathbb{R}$, we denote its derivative at $v$ by $F'(v)\in V^*$.

Given $p > 1$, let $p^*$ denote its H\"{o}lder conjugate, i.e.,
\begin{equation*}
\frac{1}{p} + \frac{1}{p^*} = 1.
\end{equation*}
We assume that the norm is strictly convex and differentiable away from the origin, and that the dual norm is differentiable away from the origin.
Let $J_p\colon V\to V^*$ denote the duality map~\cite{XR:1991} of power type $p$, defined by
\begin{equation}
\label{J_p}
J_p(v)=\left(\frac1p\|\cdot\|^p\right)'(v)=\|v\|^{p-1}j(v),
\quad v\in V,
\end{equation}
where, for $v\neq 0$, $j(v)\in V^*$ satisfies
\begin{equation}
\label{j}
(j(v),v)=\|v\|,\quad\|j(v)\|_{\ast}=1,
\end{equation}
and we set $j(0)=0$.
A useful inequality associated with the duality map is the following generalized Cauchy--Schwarz inequality, which can be proved easily using~\eqref{J_p} and~\eqref{j}:
\begin{equation}
\label{generalized_CS}
(J_p(v),w)\leq\|v\|^{p-1}\|w\|,
\quad v,w\in V.
\end{equation}
The norm assumptions give $J_p^{-1}=J_{p^*}^*$ by~\cite[Theorem~26.5]{Rockafellar:1970},
where the duality map $J_{p^*}^*\colon V^*\to V$ for the dual norm is given by
\begin{equation}
\label{J_p_star}
J_{p^*}^*(\xi)=\left(\frac1{p^*}\|\cdot\|_{\ast}^{p^*}\right)'(\xi)
=\|\xi\|_{\ast}^{p^*-1}j_*(\xi),
\quad \xi\in V^*.
\end{equation}
Here, for $\xi\neq 0$, $j_*(\xi)\in V$ satisfies
\begin{equation*}
(\xi,j_*(\xi))=\|\xi\|_{\ast},\quad\|j_*(\xi)\|=1,
\end{equation*}
and $j_*(0)=0$.
We use the same notation for the duality maps on $W^*$ and $W$.

For a differentiable and convex functional $F \colon V \to \mathbb{R}$, $D_F (v,w)$ denotes the Bregman distance associated with $F$ from $w$ to $v$:
\begin{equation*}
    D_F (v, w) = F (v) - F(w) - ( F'(w), v-w),
    \quad v, w \in V.
\end{equation*}
We also define the symmetrized Bregman distance~\cite{PH:2025} by
\begin{equation}
\label{Bregman_sym}
    D_F^{\mathrm{sym}} (v,w)
    = D_F (v,w) + D_F(w,v)
    = (F'(v) - F'(w), v - w),
    \quad v, w \in V.
\end{equation}

A differentiable and convex functional $F \colon V \to \mathbb{R}$ is said to be $(p,\mu)$-uniformly convex for some $p \geq 2$ and $\mu > 0$ if
\begin{equation*}
    D_F (v, w) \geq \frac{\mu}{p} \| v - w \|^p,
    \quad v, w \in V.
\end{equation*}
When the constant $\mu$ is not essential, we simply say that $F$ is $p$-uniformly convex, or simply uniformly convex.
Similarly, $F$ is said to be $(q,L)$-weakly smooth for some $1 < q \leq 2$ and $L > 0$ if
\begin{equation*}
    D_F (v, w) \leq \frac{L}{q} \| v - w \|^q,
    \quad v, w \in V.
\end{equation*}
When the constant $L$ is omitted, we simply say that $F$ is $q$-weakly smooth.
See~\cite{Park:2022b} for further discussion of uniform convexity and weak smoothness.

For a convex functional $F \colon V \to \mathbb{R}$, we define its Legendre--Fenchel conjugate $F^* \colon V^* \to \overline{\mathbb{R}}$ by
\begin{equation*}
    F^* (\xi) = \sup_{v \in V} \left\{ (\xi, v) - F (v) \right\},
    \quad \xi \in V^*.
\end{equation*}
An important property of the Legendre--Fenchel conjugate is that it reverses order: if $F \leq G$, then $F^* \geq G^*$; see, e.g.,~\cite{Rockafellar:1970}.

Using the Legendre--Fenchel conjugate, one can establish the following useful relationship between uniform convexity and weak smoothness.
Although this result is well known (see, e.g.,~\cite{Zalinescu:2002}), we include a proof for completeness.

% Proposition: Weak smoothness and uniform convexity
\begin{proposition}
\label{Prop:smooth}
Let $F \colon V \to \mathbb{R}$ be a uniformly convex and $(q, L)$-weakly smooth functional for some $1 < q \leq 2$ and $L > 0$.
Then its Legendre--Fenchel conjugate $F^*$ is $(q^*, L^{-(q^*-1)})$-uniformly convex with respect to the dual norm $\| \cdot \|_{\ast}$.
\end{proposition}
\begin{proof}
Since $F$ is $(q, L)$-weakly smooth, we have
\begin{equation*}
    D_F(v + w, v) \leq \frac{L}{q} \| w \|^q,
    \quad v, w \in V.
\end{equation*}
Taking the Legendre--Fenchel conjugate with respect to $w$ on both sides, we obtain
\begin{equation*}
    D_{F^*} ( \xi + F'(v), F'(v) ) \geq \frac{1}{q^* L^{q^* - 1}} \| \xi \|_{\ast}^{q^*},
    \quad v\in V,\ \xi\in V^*.
\end{equation*}
Since $F$ is uniformly convex and weakly smooth, the derivative $F'\colon V\to V^*$ is surjective~\cite{Teboulle:2018}.
Therefore, we conclude that
\begin{equation*}
    D_{F^*} ( \xi + \eta, \eta ) \geq \frac{1}{q^* L^{q^* - 1}} \| \xi \|_{\ast}^{q^*},
    \quad \xi, \eta \in V^*.
\end{equation*}
This completes the proof.
\end{proof}

% Remark: Local assumptions
\begin{remark}
\label{Rem:local_conditions}
In some applications, uniform convexity and weak smoothness hold only locally~\cite{GR:2008}.
That is, the constants $L$ and $\mu$ may depend on $v$ and $w$ (see, e.g.,~\cite{BL:1994}).
However, it is often the case that $L$ is bounded above and $\mu$ is bounded away from zero on bounded convex sets containing the relevant iterates and the solution, and our analysis only requires uniform convexity and weak smoothness on such sets.
Therefore, for simplicity, we state our results without explicitly indicating this local dependence, which leads to a clearer presentation.
A similar discussion can be found in~\cite[Remark~2.1]{TX:2002}.
On the other hand, a complete treatment of this local dependence is provided in, for example,~\cite{LP:2025b}.
\end{remark}

% Section: High-order proximal point method
\section{High-order proximal point method}
\label{Sec:PPM}
In this section, we study the convergence behavior of the high-order proximal point method, introduced in~\cite{AN:2024,Nesterov:2021,Nesterov:2023}, when the objective functional is differentiable and uniformly convex.
The three regimes are determined by the regularization power and the uniform convexity exponent, in agreement with the nonlinear proximal point analysis in~\cite{OBLP:2026}.

We consider the following model problem:
\begin{equation}
\label{model_PPM}
\min_{v \in V} F(v),
\end{equation}
where $F \colon V \to \mathbb{R}$ is a differentiable and convex functional.
We assume that~\eqref{model_PPM} admits a solution, denoted by $u \in V$.

The high-order proximal point method (see, e.g.,~\cite{AN:2024}) for solving~\eqref{model_PPM} is described in \cref{Alg:PPM}.

% Algorithm: Augmented Lagrangian method
\begin{algorithm}
\caption{High-order proximal point method for solving~\eqref{model_PPM}}
\begin{algorithmic}[]
\label{Alg:PPM}
\STATE Given $r > 1$ and $\epsilon > 0$:
\STATE Choose $u^{(0)} \in V$.
\FOR{$n=0,1,2,\dots$}
    \STATE $\displaystyle
    u^{(n+1)} = \operatornamewithlimits{\arg\min}_{v \in V} \left\{ F (v) + \frac{\epsilon}{r} \| v - u^{(n)} \|^r \right\}
    $
\ENDFOR
\end{algorithmic}
\end{algorithm}

In \cref{Alg:PPM}, the first-order optimality condition for $u^{(n+1)}$ is written as
\begin{equation}
\label{PPM_optimality}
F' (u^{(n+1)}) + \epsilon J_r ( u^{(n+1)}-u^{(n)} ) = 0,
\end{equation}
where $J_r$ was defined in~\eqref{J_p}.

When $r = 2$ and the norm is induced by an inner product, \cref{Alg:PPM} reduces to the standard proximal point method~\cite{Martinet:1970,Rockafellar:1976a}.
The convergence analysis for this case has been extensively studied in the literature; see~\cite{Kim:2021} and the references therein.
On the other hand, for the case $r > 2$, the analysis in~\cite{Nesterov:2023} shows that larger values of $r$ yield faster convergence.

Here, we impose an additional $(p, \mu)$-uniform convexity assumption on $F$, which is typical in variational problems arising in nonlinear PDEs (see, e.g.,~\cite{LP:2025a,TX:2002}).

We first consider the case $r = p$.
In this setting, we prove that \cref{Alg:PPM} achieves linear convergence of the energy error, as stated in \cref{Thm:PPM_linear}.

% Theorem: Linear convergence of PPM
\begin{theorem}
\label{Thm:PPM_linear}
In~\eqref{model_PPM}, suppose that $F$ is $(p, \mu)$-uniformly convex for some $p \geq 2$ and $\mu > 0$.
In \cref{Alg:PPM}, if $r = p$, then we have
\begin{equation}
\label{Thm1:PPM_linear}
    \frac{F(u^{(n+1)}) - F(u)}{F(u^{(n)}) - F(u)}
    \leq \frac{1}{1 + \gamma} \text{ with }
    \gamma = \frac{p}{p-1} \left( \frac{\mu}{\epsilon} \right)^{\frac{1}{p-1}} + \frac{1}{p-1} \left( \frac{\mu}{\epsilon} \right)^{\frac{p}{p-1}},
    \quad n \geq 0.
\end{equation}
Moreover, if $\epsilon < \mu/p$, then we have
\begin{equation}
\label{Thm2:PPM_linear}
    \frac{\| u^{(n+1)} - u \|}{\| u^{(n)} - u \|}
    \leq \frac{1}{\left( \frac{\mu}{p \epsilon} \right)^{\frac{1}{p-1}} - 1}, \quad n \geq 0.
\end{equation}
\end{theorem}
\begin{proof}
By the $(p, \mu)$-uniform convexity of $F$, we have
\begin{equation}
\label{Thm3:PPM_linear}
\begin{split}
F (u^{(n)})
&\geq F (u^{(n+1)}) + ( F' (u^{(n+1)}), u^{(n)} - u^{(n+1)} ) + \frac{\mu}{p} \|u^{(n)}-u^{(n+1)} \|^p \\
&\stackrel{\eqref{PPM_optimality}}{=} F (u^{(n+1)})+ \left( \epsilon+\frac{\mu}{p} \right) \|u^{(n+1)}-u^{(n)}\|^p.
\end{split}
\end{equation}
On the other hand, again by the $(p,\mu)$-uniform convexity of $F$, we get
\begin{equation}
\begin{split}
\label{Thm4:PPM_linear}
F(u^{(n+1)}) - F(u)
&\leq - (F' (u^{(n+1)}), u - u^{(n+1)} ) - \frac{\mu}{p} \| u - u^{(n+1)} \|^p \\
&\stackrel{\eqref{PPM_optimality}}{=} \epsilon ( J_p(u^{(n+1)}-u^{(n)}), u - u^{(n+1)}) - \frac{\mu}{p} \| u - u^{(n+1)} \|^p \\
&\stackrel{\eqref{generalized_CS}}{\leq} \epsilon \| u^{(n+1)} - u^{(n)} \|^{p-1} \| u - u^{(n+1)} \| - \frac{\mu}{p} \| u - u^{(n+1)} \|^p \\
&\leq \sup_{t \geq 0}\ \left\{ \epsilon \|u^{(n+1)}-u^{(n)}\|^{p-1} t - \frac{\mu}{p} t^p \right\} \\
&= \frac{p-1}{p} \epsilon^{\frac{p}{p-1}} \mu^{-\frac{1}{p-1}} \|u^{(n+1)}-u^{(n)}\|^p.
\end{split}
\end{equation}
Combining~\eqref{Thm3:PPM_linear} and~\eqref{Thm4:PPM_linear} yields~\eqref{Thm1:PPM_linear}.

To prove~\eqref{Thm2:PPM_linear}, we observe that
\begin{equation*}
\begin{split}
\frac{\mu}{p} \| u^{(n+1)} - u \|^p
&\leq ( F' (u^{(n+1)}) - F'(u), u^{(n+1)} - u) \\
&\stackrel{\eqref{PPM_optimality}}{=} - \epsilon ( J_p(u^{(n+1)} - u^{(n)}),u^{(n+1)}-u ) \\
&\stackrel{\eqref{generalized_CS}}{\leq} \epsilon\|u^{(n+1)} - u^{(n)}\|^{p-1} \| u^{(n+1)} - u \| \\
&\leq \epsilon ( \| u^{(n+1)} - u \| + \| u^{(n)} - u \| )^{p-1} \| u^{(n+1)} - u \|.
\end{split}
\end{equation*}
It follows that
\begin{equation*}
    \| u^{(n+1)} - u \| \leq \left( \frac{p \epsilon}{\mu} \right)^{\frac{1}{p-1}} ( \| u^{(n+1)} - u \| + \| u^{(n)} - u \|),
\end{equation*}
which implies~\eqref{Thm2:PPM_linear}.
\end{proof}

In \cref{Thm:PPM_linear}, we note that $\gamma$ tends to $\infty$ as $\epsilon$ tends to $0$.
Namely, the linear convergence rate of \cref{Alg:PPM} becomes arbitrarily fast as $\epsilon \to 0$.
Such parameter dependence is familiar in proximal point methods~\cite{OBLP:2026}.

Next, we consider the case $r > p$.
In this setting, we obtain superlinear convergence; see \cref{Thm:PPM_superlinear}.
Related local convergence results under growth assumptions are given in~\cite{OBLP:2026}.

% Theorem: Superlinear convergence of PPM
\begin{theorem}
\label{Thm:PPM_superlinear}
In~\eqref{model_PPM}, suppose that $F$ is $(p, \mu)$-uniformly convex for some $p \geq 2$ and $\mu > 0$.
In \cref{Alg:PPM}, if $r > p$, then $F(u^{(n)})$ converges to $F(u)$ superlinearly and satisfies
\begin{equation}
\label{Thm1:PPM_superlinear}
    \frac{F(u^{(n+1)}) - F(u)}{( F(u^{(n)}) - F(u) )^{\frac{r-1}{p-1}}}
    \leq (p-1) \left( \frac{p^{r-p} \epsilon^p}{\mu^r} \right)^{\frac{1}{p-1}},
    \quad n \geq 0.
\end{equation}
\end{theorem}
\begin{proof}
Proceeding similarly as in~\eqref{Thm3:PPM_linear}, we obtain
\begin{equation}
\label{Thm2:PPM_superlinear}
    F(u^{(n)}) \geq F(u^{(n+1)}) + \epsilon \| u^{(n+1)} - u^{(n)} \|^r + \frac{\mu}{p} \| u^{(n+1)} - u^{(n)} \|^p.
\end{equation}
On the other hand, proceeding similarly as in~\eqref{Thm4:PPM_linear}, we obtain
\begin{equation}
\label{Thm3:PPM_superlinear}
\begin{split}
F(u^{(n+1)}) - F(u)
&\mathrel{\leq} \epsilon ( J_r (u^{(n+1)}-u^{(n)}), u - u^{(n+1)}) - \frac{\mu}{p} \| u - u^{(n+1)} \|^p \\
&\stackrel{\eqref{generalized_CS}}{\leq} \sup_{t \geq 0} \left\{ \epsilon \| u^{(n+1)}-u^{(n)}\|^{r-1} t - \frac{\mu}{p} t^p \right\} \\
&= \frac{p-1}{p} \epsilon^{\frac{p}{p-1}} \mu^{-\frac{1}{p-1}} \|u^{(n+1)}-u^{(n)}\|^{\frac{p(r-1)}{p-1}}.
\end{split}
\end{equation}
Combining~\eqref{Thm2:PPM_superlinear} and~\eqref{Thm3:PPM_superlinear} yields
\begin{equation}
\label{Thm4:PPM_superlinear}
\zeta_n \geq \zeta_{n+1}
+ \left( \frac{p}{p-1} \right)^{\frac{r(p-1)}{p(r-1)}} \left( \frac{\mu^{\frac{r}{p}}}{\epsilon} \right)^{\frac{1}{r-1}} \zeta_{n+1}^{\frac{r(p-1)}{p(r-1)}}
+ \frac{1}{p} \left( \frac{p}{p-1} \right)^{\frac{p-1}{r-1}} \left( \frac{\mu^{\frac{r}{p}}}{\epsilon} \right)^{\frac{p}{r-1}} \zeta_{n+1}^{\frac{p-1}{r-1}},
\end{equation}
where $\zeta_n = F(u^{(n)}) - F(u)$.
Note that
\begin{equation*}
    0 < \frac{p-1}{r-1} < \frac{r (p-1)}{p (r-1)} < 1.
\end{equation*}
We observe that~\eqref{Thm4:PPM_superlinear} implies $\zeta_n \to 0$.
Moreover, by dropping the first two terms on the right-hand side, we obtain
\begin{equation*}
    \zeta_n \geq \frac{1}{p} \left( \frac{p}{p-1} \right)^{\frac{p-1}{r-1}} \left( \frac{\mu^{\frac{r}{p}}}{\epsilon} \right)^{\frac{p}{r-1}} \zeta_{n+1}^{\frac{p-1}{r-1}},
\end{equation*}
which yields~\eqref{Thm1:PPM_superlinear}.
\end{proof}

Finally, we consider the case $r < p$.
In this setting, we obtain sublinear convergence of \cref{Alg:PPM}, as stated in \cref{Thm:PPM_sublinear}.
The estimate uses uniform convexity; convergence under general convexity assumptions is studied in~\cite{Nesterov:2023}.

% Theorem: Sublinear convergence of PPM
\begin{theorem}
\label{Thm:PPM_sublinear}
In~\eqref{model_PPM}, suppose that $F$ is $(p, \mu)$-uniformly convex for some $p \geq 2$ and $\mu > 0$.
In \cref{Alg:PPM}, if $r < p$, then we have
\begin{equation*}
    F(u^{(n)}) - F(u)
    \leq \frac{\zeta_0}{\left( 1 + \frac{n}{\beta+1} \log ( 1 + C_{\epsilon} \zeta_0^{1/\beta} ) \right)^{\beta}},
    \quad n \geq 0,
\end{equation*}
where $\zeta_0 = F(u^{(0)}) - F(u)$ and
\begin{equation*}
    \beta = \frac{p (r-1)}{p-r}, \quad
    C_{\epsilon} = \left( \frac{p}{p-1} \right)^{\frac{r(p-1)}{p(r-1)}} \left( \frac{\mu^{\frac{r}{p}}}{\epsilon} \right)^{\frac{1}{r-1}}.
\end{equation*}
\end{theorem}
\begin{proof}
We start from~\eqref{Thm4:PPM_superlinear}, which remains valid in the case $r < p$.
Note that
\begin{equation*}
    1 < \frac{r (p-1)}{p (r-1)} < \frac{p-1}{r-1}.
\end{equation*}
From~\eqref{Thm4:PPM_superlinear}, we drop the last term on the right-hand side to obtain
\begin{equation}
\label{Thm1:PPM_sublinear}
    \zeta_n \geq \zeta_{n+1}
+ \left( \frac{p}{p-1} \right)^{\frac{r(p-1)}{p(r-1)}} \left( \frac{\mu^{\frac{r}{p}}}{\epsilon} \right)^{\frac{1}{r-1}} \zeta_{n+1}^{\frac{r(p-1)}{p(r-1)}}.
\end{equation}
Invoking~\cite[Lemma~A.1]{Nesterov:2022} together with~\eqref{Thm1:PPM_sublinear} yields the desired result.
\end{proof}

In \cref{Thm:PPM_sublinear}, we observe that the constant $C_{\epsilon}$ tends to $\infty$ as $\epsilon \to 0$.
Thus, a smaller $\epsilon$ also improves the outer energy error bound in the sublinear regime.

% Section: High-order augmented Lagrangian method
\section{High-order augmented Lagrangian method}
\label{Sec:ALM}
In this section, we study a high-order version of the augmented Lagrangian method~\cite{OBLP:2026} for solving convex optimization problems with linear constraints.
By drawing a connection between the high-order proximal point method introduced in \cref{Sec:PPM} and the high-order augmented Lagrangian method, we analyze the convergence of the method (cf.~\cite{JPX:2025}).
We assume that the primal and dual problems admit solutions and that each augmented subproblem attains its minimum.

We consider the model constrained convex optimization problem~\eqref{model_ALM}.
The high-order augmented Lagrangian method for solving~\eqref{model_ALM} is presented in \cref{Alg:ALM}.
Note that, when $r = 2$ and the norm on $W^*$ is induced by an inner product, \cref{Alg:ALM} reduces to the classical augmented Lagrangian method~\eqref{ALM_intro}.

% Algorithm: Augmented Lagrangian method
\begin{algorithm}
\caption{High-order augmented Lagrangian method for solving~\eqref{model_ALM}}
\begin{algorithmic}[]
\label{Alg:ALM}
\STATE Given $r > 1$ and $\epsilon > 0$:
\STATE Choose $\lambda^{(0)} \in W^*$.
\FOR{$n=0,1,2,\dots$}
    \STATE $\displaystyle
    u^{(n+1)} = \operatornamewithlimits{\arg\min}_{v \in V} \left\{ F (v) + ( \lambda^{(n)}, Bv - g ) + \frac{1}{r^* \epsilon^{r^* - 1}} \| Bv - g \|_{\ast}^{r^*} \right\}
    $
    \STATE $\displaystyle
    \lambda^{(n+1)} = \operatornamewithlimits{\arg\min}_{\sigma \in W^*} \left\{ - (\sigma, Bu^{(n+1)} - g) + \frac{\epsilon}{r} \| \sigma - \lambda^{(n)} \|^r \right\}
    $
\ENDFOR
\end{algorithmic}
\end{algorithm}

For many commonly used norms, the $\lambda^{(n+1)}$-subproblem in \cref{Alg:ALM} can be solved explicitly, as in the conventional augmented Lagrangian method~\eqref{ALM_intro}.
Note that the first-order optimality condition for the $\lambda^{(n+1)}$-subproblem is given by
\begin{equation}
\label{lambda_optimality}
- (B u^{(n+1)} - g) + \epsilon J_r (\lambda^{(n+1)} - \lambda^{(n)}) = 0.
\end{equation}
From~\eqref{lambda_optimality}, we readily obtain
\begin{equation}
\begin{split}
\label{dual_update_unstable}
    \lambda^{(n+1)}
    &= \lambda^{(n)} + J_{r^*}^* ( \epsilon^{-1} (Bu^{(n+1)} - g) ) \\
    &=\lambda^{(n)} + \epsilon^{- (r^* - 1)} \| Bu^{(n+1)} - g \|_{\ast}^{r^*-1} j_* (Bu^{(n+1)} - g),
\end{split}
\end{equation}
where $J_{r^*}^*$ was defined in~\eqref{J_p_star}.
Although~\eqref{dual_update_unstable} appears reasonable, it may suffer from numerical instability, especially when $\epsilon$ is small and $r^*$ is close to $1$.

To analyze the convergence of the high-order augmented Lagrangian method~(\cref{Alg:ALM}), we first consider the following Fenchel--Rockafellar dual problem of~\eqref{model_ALM} (see, e.g.,~\cite{CP:2016,JPX:2025} for details):
\begin{equation}
\label{dual_ALM}
\min_{\sigma \in W^*} \left\{ E_{\mathrm{d}} (\sigma) := F^* ( -B^t \sigma) + (g, \sigma) \right\}.
\end{equation}
Here, $E_{\mathrm{d}}\colon W^*\to\overline{\mathbb{R}}$ is the dual energy, and a solution of~\eqref{dual_ALM} is denoted by $\lambda\in W^*$.
The equivalence between the proximal point method for solving~\eqref{dual_ALM} and the augmented Lagrangian method for solving~\eqref{model_ALM} has been discussed in the literature; see, e.g.,~\cite{PH:2025,Rockafellar:1976b}.

Given $r > 1$ and $\epsilon > 0$, the high-order proximal point method (\cref{Alg:PPM}) applied to~\eqref{dual_ALM} takes the form
\begin{equation}
\label{PPM_dual}
\sigma^{(n+1)} = \operatornamewithlimits{\arg\min}_{\sigma \in W^*} \left\{ F^* (-B^t \sigma) + (g, \sigma) + \frac{\epsilon}{r} \| \sigma - \sigma^{(n)} \|^r \right\},
\quad n \geq 0.
\end{equation}
In \cref{Thm:equiv}, using a dualization argument introduced in~\cite{JPX:2025}, we verify the equivalence between \cref{Alg:ALM} and~\eqref{PPM_dual}.

% Theorem: Equivalence between PPM and ALM
\begin{theorem}
\label{Thm:equiv}
The high-order augmented Lagrangian method~(\cref{Alg:ALM}) for solving~\eqref{model_ALM} and the high-order proximal point method~\eqref{PPM_dual} for solving~\eqref{dual_ALM} are equivalent in the following sense: if $\lambda^{(0)} = \sigma^{(0)}$, then we have $\lambda^{(n)} = \sigma^{(n)}$ for all $n \geq 0$.
\end{theorem}
\begin{proof}
Fix any $n \geq 0$ and suppose inductively that $\lambda^{(n)}=\sigma^{(n)}$.
We define the functional $\phi_n \colon W^* \to \mathbb{R}$ by
\begin{equation*}
    \phi_n (\sigma) = (g, \sigma) + \frac{\epsilon}{r} \| \sigma - \sigma^{(n)} \|^r,
    \quad \sigma \in W^*.
\end{equation*}
Its Legendre--Fenchel conjugate $\phi_n^* \colon W \to \mathbb{R}$ is given by
\begin{equation*}
    \phi_n^* (w) = \frac{1}{r^* \epsilon^{r^* - 1}} \| w - g \|_{\ast}^{r^*} + (w - g, \sigma^{(n)}),
    \quad w \in W.
\end{equation*}
Therefore, by Fenchel--Rockafellar duality~\cite{CP:2016,JPX:2025}, we obtain the following dual problem of~\eqref{PPM_dual}:
\begin{equation}
\label{Thm1:equiv}
    u^{(n+1)} = \operatornamewithlimits{\arg\min}_{v \in V} \left\{ F (v) + ( \lambda^{(n)}, Bv - g ) + \frac{1}{r^* \epsilon^{r^* - 1}} \| Bv - g \|_{\ast}^{r^*} \right\},
\end{equation}
which coincides with the $u^{(n+1)}$-subproblem in \cref{Alg:ALM}.
Moreover, the primal--dual relation between~\eqref{PPM_dual} and~\eqref{Thm1:equiv} is given by
\begin{subequations}
\begin{align}
\label{primal_dual_1}
    -B^t \sigma^{(n+1)} &= F' (u^{(n+1)}), \\
\label{primal_dual_2}
    B u^{(n+1)} &= \phi_n' (\sigma^{(n+1)}).
\end{align}
\end{subequations}
Note that~\eqref{primal_dual_2} is equivalent to the first-order optimality condition~\eqref{lambda_optimality} for the $\lambda^{(n+1)}$-subproblem in \cref{Alg:ALM} after identifying $\sigma^{(n+1)}$ with $\lambda^{(n+1)}$.
This completes the proof.
\end{proof}

Thanks to the equivalence established in \cref{Thm:equiv}, by invoking the convergence results for the high-order proximal point method given in \cref{Thm:PPM_linear,Thm:PPM_superlinear,Thm:PPM_sublinear}, we readily obtain the following convergence result for the high-order augmented Lagrangian method.

% Corollary: Equivalence between PPM and ALM
\begin{corollary}
\label{Cor:equiv}
In~\eqref{model_ALM} and~\eqref{dual_ALM}, suppose that $E_{\mathrm{d}}$ is $(p, \mu)$-uniformly convex for some $p \geq 2$ and $\mu > 0$.
Then in \cref{Alg:ALM}, the following hold:
\begin{enumerate}[(a)]
\item If $r = p$, then we have
\begin{equation*}
\resizebox{0.9\textwidth}{!}{$\displaystyle
    \frac{E_{\mathrm{d}} (\lambda^{(n+1)}) - E_{\mathrm{d}} ( \lambda)}{E_{\mathrm{d}} ( \lambda^{(n)}) - E_{\mathrm{d}} ( \lambda )}
    \leq \frac{1}{1 + \gamma}
     \text{ with }
    \gamma = \frac{p}{p-1} \left( \frac{\mu}{\epsilon} \right)^{\frac{1}{p-1}} + \frac{1}{p-1} \left( \frac{\mu}{\epsilon} \right)^{\frac{p}{p-1}},
    \quad n \geq 0.
    $}
\end{equation*}
\item If $r > p$, then $E_{\mathrm{d}} (\lambda^{(n)})$ converges to $E_{\mathrm{d}} (\lambda)$ superlinearly and satisfies
\begin{equation*}
    \frac{E_{\mathrm{d}} (\lambda^{(n+1)}) - E_{\mathrm{d}} ( \lambda)}{( E_{\mathrm{d}} (\lambda^{(n)}) - E_{\mathrm{d}} ( \lambda) )^{\frac{r-1}{p-1}}}
    \leq (p-1) \left( \frac{p^{r-p} \epsilon^p}{\mu^r} \right)^{\frac{1}{p-1}},
    \quad n \geq 0.
\end{equation*}
\item If $r < p$, then we have
\begin{equation*}
    E_{\mathrm{d}} (\lambda^{(n)}) - E_{\mathrm{d}} ( \lambda)
    \leq \frac{\zeta_0}{\left( 1 + \frac{n}{\beta+1} \log ( 1 + C_{\epsilon} \zeta_0^{1/\beta} ) \right)^{\beta}},
    \quad n \geq 0,
\end{equation*}
where $\zeta_0 = E_{\mathrm{d}} (\lambda^{(0)}) - E_{\mathrm{d}} ( \lambda)$ and
\begin{equation*}
    \beta = \frac{p (r-1)}{p-r}, \quad
    C_{\epsilon} = \left( \frac{p}{p-1} \right)^{\frac{r(p-1)}{p(r-1)}} \left( \frac{\mu^{\frac{r}{p}}}{\epsilon} \right)^{\frac{1}{r-1}}.
\end{equation*}
\end{enumerate}
\end{corollary}

% Remark: Equivalence between PPM and ALM
\begin{remark}
\label{Rem:equiv}
In \cref{Cor:equiv}, since $E_{\mathrm{d}}$ is $(p, \mu)$-uniformly convex and $\lambda$ is a minimizer of $E_{\mathrm{d}}$, we have
\begin{equation}
\label{Rem1:equiv}
E_{\mathrm{d}} ( \lambda^{(n)}) - E_{\mathrm{d}} ( \lambda) \geq \frac{\mu}{p} \| \lambda^{(n)} - \lambda \|^p,
\quad n \geq 0.
\end{equation}
Hence, combining~\eqref{Rem1:equiv} with \cref{Cor:equiv} yields norm convergence of the dual sequence $\{ \lambda^{(n)} \}$ to $\lambda$.
\end{remark}

% Remark: Non-uniformly convex case
\begin{remark}
\label{Rem:uniform}
Since the high-order proximal point method is known to be convergent even without a uniform convexity assumption (see~\cite{AN:2024,Nesterov:2021,Nesterov:2023}), one can similarly deduce sublinear convergence of the high-order augmented Lagrangian method presented in \cref{Alg:ALM}, analogous to \cref{Cor:equiv}(c), without assuming that $E_{\mathrm{d}}$ is uniformly convex.
We omit the details.
\end{remark}

While \cref{Cor:equiv} already provides a desirable convergence result for \cref{Alg:ALM}, in many cases it is also important to verify whether the primal sequence $\{ u^{(n)} \}$ converges.
In \cref{Thm:primal}, we present a primal convergence result in terms of the symmetrized Bregman distance~\eqref{Bregman_sym}, following an argument similar to that in~\cite[Theorem~A.4]{PH:2025}.

% Theorem: Primal convergence
\begin{theorem}
\label{Thm:primal}
In~\eqref{model_ALM} and~\eqref{dual_ALM}, suppose that $E_{\mathrm{d}}$ is $(p, \mu)$-uniformly convex for some $p \geq 2$ and $\mu > 0$.
Then in \cref{Alg:ALM}, we have
\begin{equation*}
    D_F^{\mathrm{sym}} (u^{(n+1)}, u)
    \leq (r-1) (\zeta_n - \zeta_{n+1}) + \frac{\epsilon}{r} \left( \frac{p}{\mu} \right)^{\frac{r}{p}} \zeta_{n+1}^{\frac{r}{p}},
    \quad n \geq 0,
\end{equation*}
where $\zeta_n = E_{\mathrm{d}} ( \lambda^{(n)}) - E_{\mathrm{d}} (\lambda)$.
\end{theorem}
\begin{proof}
Note that the optimality conditions for~\eqref{model_ALM} read
\begin{align}
    \label{model_ALM_optimality_1}
    F'(u) + B^t \lambda &= 0, \\
    Bu &= g. \notag
\end{align}
It follows that
\begin{equation}
\label{Thm1:primal}
\begin{split}
    D_F^{\mathrm{sym}} (u^{(n+1)}, u)
    &\stackrel{\eqref{Bregman_sym}}{=} (F' (u^{(n+1)}) - F' (u), u^{(n+1)} - u) \\
    &\stackrel{\eqref{primal_dual_1}, \eqref{model_ALM_optimality_1}}{=} ( - B^t (\lambda^{(n+1)} - \lambda), u^{(n+1)} - u ) \\
    &\stackrel{\eqref{lambda_optimality}}{=} - \epsilon (\lambda^{(n+1)} - \lambda, J_r (\lambda^{(n+1)} - \lambda^{(n)}) ) \\
    &\stackrel{\eqref{generalized_CS}}{\leq} \epsilon \| \lambda^{(n+1)} - \lambda \| \| \lambda^{(n+1)} - \lambda^{(n)} \|^{r-1} \\
    &\leq \frac{\epsilon}{r} \| \lambda^{(n+1)} - \lambda \|^r + \frac{(r-1) \epsilon}{r} \| \lambda^{(n+1)} - \lambda^{(n)} \|^r,
\end{split}
\end{equation}
where the last two inequalities are due to the Cauchy--Schwarz and Young inequalities.
Meanwhile, from the minimization property of $\lambda^{(n+1)}$ characterized by~\eqref{PPM_dual} and \cref{Thm:equiv}, we have
\begin{equation}
\label{Thm2:primal}
E_{\mathrm{d}} ( \lambda^{(n)} ) - E_{\mathrm{d}} ( \lambda^{(n+1)} ) \geq \frac{\epsilon}{r} \| \lambda^{(n+1)} - \lambda^{(n)} \|^r.
\end{equation}
In addition, by the $(p, \mu)$-uniform convexity of $E_{\mathrm{d}}$, we have
\begin{equation}
\label{Thm3:primal}
E_{\mathrm{d}} (\lambda^{(n+1)}) - E_{\mathrm{d}} (\lambda) \geq \frac{\mu}{p} \| \lambda^{(n+1)} - \lambda \|^p.
\end{equation}
Combining~\eqref{Thm1:primal},~\eqref{Thm2:primal}, and~\eqref{Thm3:primal} yields the desired result.
\end{proof}

We obtain \cref{Cor:primal} as a straightforward consequence of \cref{Thm:primal}.

% Corollary: Primal convergence
\begin{corollary}
\label{Cor:primal}
In~\eqref{model_ALM} and~\eqref{dual_ALM}, suppose that $E_{\mathrm{d}}$ is $(p, \mu)$-uniformly convex for some $p \geq 2$ and $\mu > 0$.
Then in \cref{Alg:ALM}, the following hold:
\begin{enumerate}[(a)]
\item If $r = p$, then we have
\begin{equation*}
D_F^{\mathrm{sym}} (u^{(n)}, u)
\leq \left( p-1 + \frac{\epsilon}{\mu} \right) \left(\frac{1}{1+\gamma} \right)^{n-1} \zeta_0,
\quad n \geq 1,
\end{equation*}
where $\zeta_n = E_{\mathrm{d}} (\lambda^{(n)}) - E_{\mathrm{d}} (\lambda)$, and $\gamma$ was given in \cref{Cor:equiv}.
\item If $r > p$, then we have
\begin{equation*}
    D_F^{\mathrm{sym}} (u^{(n)}, u) = O ( \zeta_{n-1} ).
\end{equation*}
That is, the primal error is bounded by a superlinearly convergent sequence.
\item If $r < p$, then we have
\begin{equation*}
    D_F^{\mathrm{sym}} (u^{(n)}, u) = O(n^{-\frac{r(r-1)}{p-r}}).
\end{equation*}
That is, $D_F^{\mathrm{sym}} (u^{(n)}, u)$ converges to $0$ sublinearly.
\end{enumerate}
\end{corollary}

We now give sufficient conditions for $E_{\mathrm{d}}$ to be uniformly convex.

% Proposition: Uniform convexity of $E_{\mathrm{d}}$
\begin{proposition}
\label{Prop:E_d}
In~\eqref{model_ALM} and~\eqref{dual_ALM}, suppose that $F$ is uniformly convex and $(p^*, L)$-weakly smooth  for some $p \geq 2$ and $L > 0$, and that $B$ is surjective.
Then $E_{\mathrm{d}}$ is $(p, \beta_B^p L^{-(p-1)})$-uniformly convex, where $\beta_B$ is defined by
\begin{equation*}
\beta_B = \inf_{\sigma\in W^*,\ \sigma\neq0} \frac{\| B^t \sigma \|_{\ast}}{\| \sigma \|} > 0.
\end{equation*}
\end{proposition}
\begin{proof}
Note that $\beta_B > 0$ follows from the surjectivity of $B$.
Since $F$ is uniformly convex and $(p^*, L)$-weakly smooth, it follows from \cref{Prop:smooth} that $F^*$ is $(p, L^{-(p-1)})$-uniformly convex with respect to the dual norm $\| \cdot \|_{\ast}$.
For any $\sigma, \nu \in W^*$ with $\sigma \neq \nu$, we have
\begin{equation*}
    D_{E_{\mathrm{d}}} (\sigma, \nu)
    = D_{F^*} (-B^t \sigma, -B^t \nu)
    \geq \frac{1}{p L^{p-1}} \| B^t (\sigma - \nu) \|_{\ast}^p
    \geq \frac{\beta_B^p}{p L^{p-1}} \| \sigma - \nu \|^p,
\end{equation*}
which completes the proof.
\end{proof}

% Remark: Uniform convexity of $E_{\mathrm{d}}$
\begin{remark}
\label{Rem:E_d}
While the estimate for the uniform convexity constant of $E_{\mathrm{d}}$ in \cref{Prop:E_d} is sufficient for our purposes, we note that it is not sharp; for the special case of linear saddle point problems, compare the estimates in Theorem~4.1 and Corollary~4.2 of~\cite{Park:2025}.
\end{remark}

% Section: Computational aspects
\section{Computational aspects}
\label{Sec:Computation}
In this section, we discuss several computational issues of the high-order augmented Lagrangian method.
The method can achieve fast convergence rates when the order $r$ and the penalty parameter $\epsilon$ are chosen appropriately.
We therefore examine the computational cost of each iteration, which consists of solving the primal and dual subproblems.

% Subsection: Numerical stability
\subsection{Numerical stability of dual subproblems}

\begin{figure}
\centering\includegraphics[width=\textwidth]{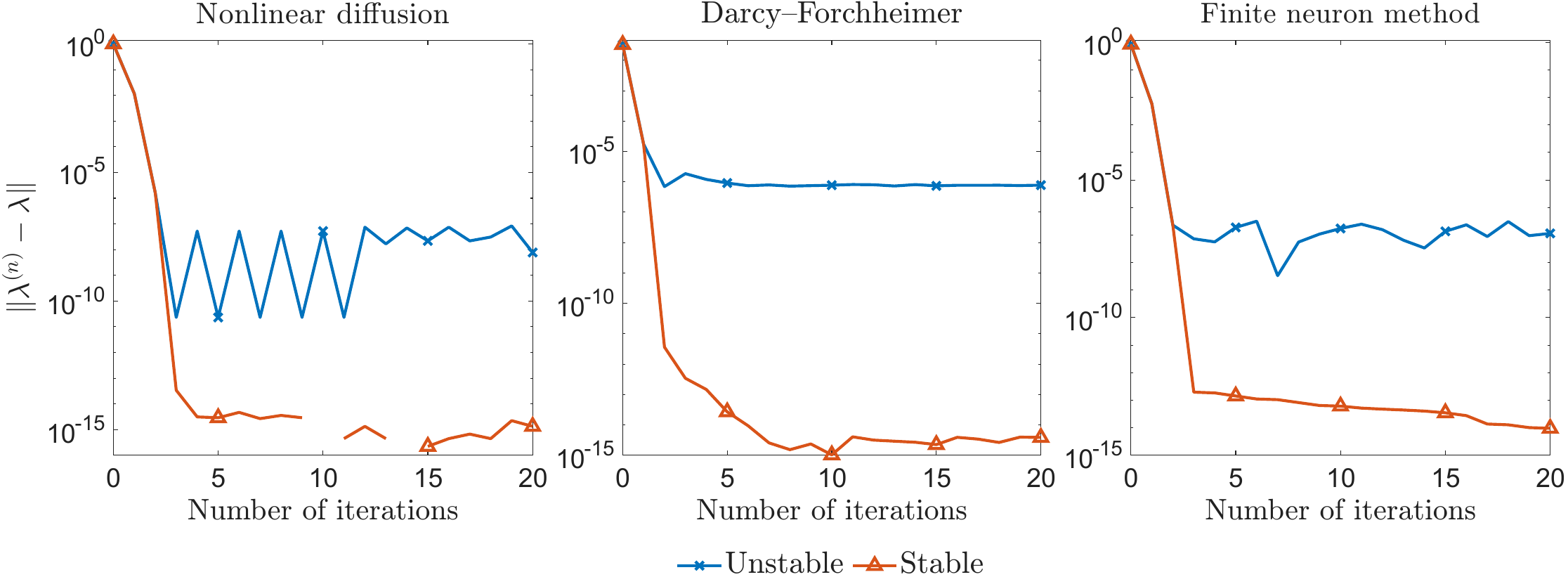}
\caption{Multiplier norm error for the high-order augmented Lagrangian method~(\cref{Alg:ALM}) with $r=3$, $\epsilon=10^{-2}$, and $h=2^{-5}$. The legends ``Unstable'' and ``Stable'' indicate separate iterations using~\eqref{dual_update_unstable} and~\eqref{dual_update_stable}, respectively. We use the absolute value for the scalar multipliers in nonlinear diffusion~\eqref{mean_energy} and the finite neuron method~\eqref{sLap_FNM}, and the $L^2$-norm for the pressure in the Darcy--Forchheimer model~\eqref{Forchheimer_opt}. The finite neuron problem uses $s=3$ and $f(x)=1+x$.}
\label{Fig:unstable}
\end{figure}

As studied in \cref{Sec:ALM}, the convergence rate of the high-order augmented Lagrangian method can be improved by choosing a large order $r$ and a small penalty parameter $\epsilon$.
However, such a choice may lead to numerical instability.
Indeed, in the dual update formula~\eqref{dual_update_unstable}, small errors in the constraint residual can be amplified when $r>2$ and $\epsilon$ is small.
In this regime, $r^*-1=(r-1)^{-1}$ lies in $(0,1)$, and the mapping $t\mapsto\operatorname{sign}(t)|t|^{r^*-1}$ has an unbounded derivative near $t=0$, which can cause numerical instability when the constraint residual is close to zero.

Fortunately, there exists a more numerically stable alternative to the dual update formula~\eqref{dual_update_unstable}, which can be employed even for large $r$ and small $\epsilon$, at the cost of solving a certain linear system.
To derive this alternative formulation, we begin by rewriting~\eqref{dual_update_unstable} as
\begin{equation}
\label{dual_update_stable_1}
    \lambda^{(n+1)} = \lambda^{(n)} + \epsilon^{- (r^* - 1)} J_{r^*}^* ( B u^{(n+1)} - g),
\end{equation}
where $J_{r^*}^*\colon W\to W^*$ is the duality map for the dual norm on $W$.
Meanwhile, the optimality condition for $u^{(n+1)}$ reads
\begin{equation*}
    F'(u^{(n+1)}) + B^t \lambda^{(n)} + \epsilon^{- (r^* - 1)} B^t J_{r^*}^* (B u^{(n+1)} - g) = 0.
\end{equation*}
Using the Riesz representatives associated with inner products on $V$ and $W$, and assuming that $B$ is surjective, we have
\begin{equation}
\label{dual_update_stable_2}
    J_{r^*}^* (B u^{(n+1)} - g)
    = - \epsilon^{r^* - 1} (B B^t )^{-1} B ( F' ( u^{(n+1)} ) + B^t \lambda^{(n)} ).
\end{equation}
Combining~\eqref{dual_update_stable_1} and~\eqref{dual_update_stable_2}, it follows that
\begin{equation}
\label{dual_update_stable}
    \lambda^{(n+1)}
    = - (B B^t)^{-1} B F'(u^{(n+1)}).
\end{equation}
This formulation avoids the explicit evaluation of the power of the constraint residual and therefore avoids this source of numerical instability even when $r$ is large and $\epsilon$ is small.
The formulation~\eqref{dual_update_stable} is particularly useful when linear systems involving $B B^t$ can be solved efficiently.

To illustrate the effect of the numerically stable formulation~\eqref{dual_update_stable}, we present in \cref{Fig:unstable} the norm error $\| \lambda^{(n)} - \lambda \|$ of the dual variable along the iterations of the high-order augmented Lagrangian method.
The dual update is performed using either~\eqref{dual_update_unstable} or~\eqref{dual_update_stable}; see \cref{Sec:Applications} for further details of the numerical experiments corresponding to \cref{Fig:unstable}.
While the two approaches exhibit similar behavior in the initial iterations, the method using~\eqref{dual_update_unstable} stagnates or oscillates after a certain point.
In contrast, the error produced by the method using~\eqref{dual_update_stable} decreases to a much smaller level, indicating that~\eqref{dual_update_stable} is significantly more numerically stable than~\eqref{dual_update_unstable}.

% Subsection: Nearly semicoercive primal subproblems
\subsection{Nearly semicoercive primal subproblems}
Next, we turn our attention to the primal subproblem in \cref{Alg:ALM}, which is another critical component of the overall algorithm.
For small- or medium-scale applications, the primal subproblem can be solved efficiently under suitable regularity assumptions using superlinearly convergent second-order methods, such as the Newton method.
However, for large-scale applications, this becomes more subtle, since solving the associated Hessian systems is generally computationally expensive.

In~\eqref{model_ALM}, because the operator $B$ typically has a nontrivial kernel, the $u^{(n+1)}$-subproblem in \cref{Alg:ALM} becomes nearly semicoercive as $\epsilon \to 0$.
More precisely, the semicoercive term involving the dual norm of $Bv - g$, which is flat along the kernel of $B$, becomes dominant as $\epsilon \to 0$; see~\cite{LP:2025b} for a rigorous mathematical definition of near semicoercivity.
When $F$ is quadratic, $r=2$, and the constraint-space norm is induced by an inner product, the subproblem is a nearly singular linear system~\cite{LWXZ:2007}, whose condition number grows as $\epsilon$ tends to zero and causes standard iterative solvers to perform poorly.
Therefore, parameter-robust numerical strategies for solving the primal subproblems, especially in large-scale settings, are essential for maintaining the computational efficiency of the high-order augmented Lagrangian method.

In~\cite{LWXZ:2007}, parameter-robust successive subspace correction methods were studied for nearly singular linear systems arising from augmented Lagrangian formulations of linear saddle point problems.
Parallel methods for nearly singular linear systems were studied in~\cite{WZ:2014}.

Subspace correction methods~\cite{TX:2002,Xu:1992} are based on a space decomposition
\begin{equation*}
    V = \sum_{j=1}^J V_j,
\end{equation*}
where each $V_j$ is a subspace of $V$.
For a fixed outer iteration $n$, write
\begin{equation*}
\Phi_n(v)=F(v)+(\lambda^{(n)},Bv-g)
+\frac{1}{r^*\epsilon^{r^*-1}}\|Bv-g\|_{\ast}^{r^*}.
\end{equation*}
Starting from an inner iterate $v^{(k)}$, a successive subspace correction sweep sets $v^{(k,0)}=v^{(k)}$ and computes
\begin{equation}
\label{subspace_sweep}
\begin{aligned}
w_j^{(k)}&\in\operatornamewithlimits{\arg\min}_{w\in V_j}
\Phi_n(v^{(k,j-1)}+w), \\
v^{(k,j)}&=v^{(k,j-1)}+w_j^{(k)},
\end{aligned}
\quad j=1,\ldots,J,
\end{equation}
followed by $v^{(k+1)}=v^{(k,J)}$.
A central structural condition for parameter-robust convergence is given by
\begin{equation}
\label{kernel_decomposition}
    \mathcal{N} = \sum_{j=1}^J (V_j \cap \mathcal{N}),
    \quad \text{where} \quad
    \mathcal{N} = \operatorname{ker} B.
\end{equation}
In~\cite{LP:2025b,LWXZ:2007}, this kernel decomposition is combined with quantitative stability estimates whose constants are independent of $\epsilon$.
In particular, domain decomposition and multigrid methods satisfying these estimates exhibit uniform convergence independent of $\epsilon$; see~\cite{LWC:2009,WZ:2014} for concrete examples.

More recently, the theory of parameter-robust parallel subspace correction methods has been extended to nearly semicoercive convex variational problems in~\cite{LP:2025b}, where near semicoercivity of convex problems directly generalizes the notion of near singularity in linear systems.
One may refer to~\cite{PH:2025} for an illustrative application to the Darcy--Forchheimer flow, a nonlinear system of PDEs describing fluid flow in porous media.
Motivated by these results, we use successive subspace correction methods for the nonlinear augmented subproblems and examine their dependence on $\epsilon$ numerically in \cref{Sec:Applications}.

We remark that the design of parameter-robust subspace correction methods must be carried out on an application-by-application basis, since the construction of a stable space decomposition satisfying~\eqref{kernel_decomposition} is problem dependent.
That is, different applications may require different strategies for designing suitable space decompositions.

\subsection{Multilevel subspace correction with FAS}
\label{Sec:FAS}
We now consider the multilevel case with nested discretization spaces
\begin{equation*}
V_1\subset\cdots\subset V_J=V,
\quad V_j=\sum_{k=1}^{n_j}V_{j,k}.
\end{equation*}
This gives the multilevel space decomposition
\begin{equation*}
V=\sum_{j=1}^J\sum_{k=1}^{n_j}V_{j,k},
\end{equation*}
where $j$ indexes the levels and $k$ indexes the local subspaces on each level.
The exact local problem in~\eqref{subspace_sweep} requires evaluating $\Phi_n$ at the fine-level iterate, even when $V_{j,k}$ belongs to a coarse level.
To reduce this cost, we use the full approximation scheme (FAS)~\cite{Brandt:1977,BF:2024}, following~\cite[Section~4.4]{PH:2025}.

Let $\Pi_j\colon V\to V_j$ be the restriction to level $j$, specified for each application, and let $\Phi_{n,j}$ denote the augmented energy discretized on level $j$.
For a current iterate $v^{\mathrm{old}}\in V$, the FAS problem on level $j$ is given by
\begin{equation}
\label{FAS_local}
\min_{w\in V_j}\left\{
\Phi_{n,j}(\Pi_jv^{\mathrm{old}}+w)
+(\Phi_n'(v^{\mathrm{old}}),w)
-(\Phi_{n,j}'(\Pi_jv^{\mathrm{old}}),w)
\right\}.
\end{equation}
At $w=0$, the derivative of the corrected coarse energy agrees with the restriction of $\Phi_n'(v^{\mathrm{old}})$ to $V_j$.
Starting from $w=0$, we apply~\eqref{subspace_sweep} to this energy using the subspaces $V_{j,k}$.
The two linear terms are fixed during the coarse solve, so successive local corrections use only coarse-level data.
At the finest level, $\Pi_J=I$ and $\Phi_{n,J}=\Phi_n$, so the linear correction vanishes.

We apply this construction recursively to the corrected energy on each level.
A V-cycle uses one successive local sweep before restriction and one after prolongation on each finer level, and one sweep on the coarsest level.

% Section: Applications
\section{Applications}
\label{Sec:Applications}
In this section, we present various applications of the high-order augmented Lagrangian method to problems arising in computational mathematics, together with relevant numerical results.
We use multilevel subspace correction methods with the FAS problems~\eqref{FAS_local} to treat mean, divergence, and boundary constraints.

In all numerical experiments involving the high-order augmented Lagrangian method~(\cref{Alg:ALM}), we used zero initial guesses; that is, we set $u^{(0)} = 0$ and $\lambda^{(0)} = 0$.
We perform $20$ iterations in each application.
For the multilevel solvers, each inner iteration is one FAS V-cycle.
The inner iteration is stopped when
\begin{equation}
\label{inner_stopping}
\frac{\|v^{(m+1)}-v^{(m)}\|}{\|v^{(m+1)}\|}<10^{-12},
\end{equation}
where $m$ denotes the inner iteration index.

The one-dimensional minimization problems are solved by the Newton method safeguarded by bisection~\cite[Section~9.4]{PTVF:2007}.
For error evaluation, we compute discrete reference solutions independently from the constrained optimality equations using the Newton method.

\subsection{Mean-constrained nonlinear diffusion}
\label{Sec:Mean}

\begin{figure}
\centering\includegraphics[width=\textwidth]{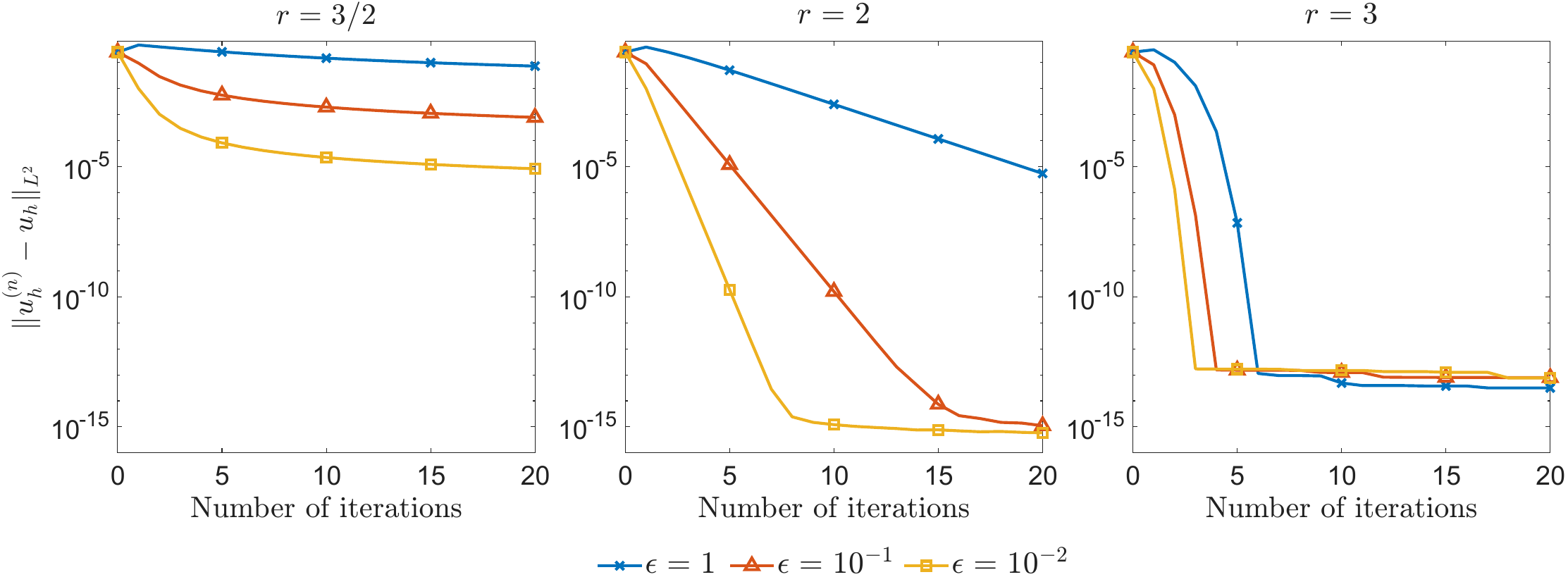}
\caption{$L^2$-norm error $\|u_h^{(n)}-u_h\|_{L^2(\Omega)}$ of the high-order augmented Lagrangian method~(\cref{Alg:ALM}) for solving the nonlinear diffusion problem~\eqref{mean_energy} with $h=2^{-5}$. Each augmented subproblem is solved by the multilevel \texttt{Nodal + Pair} method.}
\label{Fig:mean_outer}
\end{figure}
\begin{table}
\centering
\small
\caption{Average numbers of FAS V-cycles over $20$ iterations of the high-order augmented Lagrangian method~(\cref{Alg:ALM}) for solving the nonlinear diffusion problem~\eqref{mean_energy}. A dash indicates that a subproblem reached the limit of $5\times10^2$ cycles.}
\label{Tab:mean_subspace}
\begin{tabular}{crlrrrrr}
\hline
$h$ & $r$ & Subspaces & $\epsilon=1$ & $10^{-2}$ & $10^{-4}$ & $10^{-6}$ & $10^{-8}$ \\
\hline
$2^{-5}$ & $3/2$ & \texttt{Nodal} & 64.5 & 41.1 & -- & -- & -- \\
 &  & \texttt{Nodal + Pair} & 67.8 & 42.5 & 18.4 & 4.0 & 3.5 \\
 & 2 & \texttt{Nodal} & 46.8 & 5.2 & 49.7 & -- & -- \\
 &  & \texttt{Nodal + Pair} & 49.3 & 3.6 & 4.5 & 3.8 & 3.4 \\
 & 3 & \texttt{Nodal} & -- & -- & -- & -- & -- \\
 &  & \texttt{Nodal + Pair} & 9.4 & 4.0 & 4.0 & 3.6 & 3.4 \\
\hline
$2^{-6}$ & $3/2$ & \texttt{Nodal} & 64.4 & 41.0 & -- & -- & -- \\
 &  & \texttt{Nodal + Pair} & 67.8 & 42.5 & 18.4 & 4.0 & 3.4 \\
 & 2 & \texttt{Nodal} & 46.8 & 5.2 & 41.8 & -- & -- \\
 &  & \texttt{Nodal + Pair} & 49.3 & 3.6 & 4.3 & 3.8 & 3.2 \\
 & 3 & \texttt{Nodal} & -- & -- & -- & -- & -- \\
 &  & \texttt{Nodal + Pair} & 9.4 & 3.9 & 3.9 & 3.6 & 3.4 \\
\hline
$2^{-7}$ & $3/2$ & \texttt{Nodal} & 64.4 & 41.0 & -- & -- & -- \\
 &  & \texttt{Nodal + Pair} & 67.8 & 42.5 & 18.4 & 4.1 & 3.4 \\
 & 2 & \texttt{Nodal} & 46.8 & 5.2 & 39.4 & -- & -- \\
 &  & \texttt{Nodal + Pair} & 49.2 & 3.6 & 4.2 & 3.8 & 3.5 \\
 & 3 & \texttt{Nodal} & -- & -- & -- & -- & -- \\
 &  & \texttt{Nodal + Pair} & 9.4 & 3.9 & 3.9 & 3.6 & 2.8 \\
\hline
\end{tabular}
\end{table}

On $\Omega=(0,1)^2$, we consider the following mean-constrained nonlinear diffusion problem~\cite{LSL:2022}:
\begin{equation}
\label{mean_energy}
\begin{aligned}
&\min_{v\in H^1(\Omega)}\left\{F(v):=\int_\Omega\left(\frac12|\nabla v|^2+\frac12v^2+\frac14v^4-fv\right)\,dx\right\} \\
&\text{subject to}\quad\int_\Omega v\,dx=g.
\end{aligned}
\end{equation}
The data $f$ and $g$ are chosen so that the exact solution and Lagrange multiplier are given by
\begin{equation*}
u(x,y)=\frac15+\frac3{10}\cos(\pi x)\cos(\pi y),
\quad\lambda=1.
\end{equation*}
We discretize~\eqref{mean_energy} by continuous piecewise linear finite elements on nested uniform triangulations with mesh widths $h_j=2^{-j}$, $1\leq j\leq J$, and write $h=h_J$.
Let $V_j$ be the finite element space on level $j$ with nodal basis $\{\phi_{j,k}\}$, and write $V_h=V_J$.

We next construct a decomposition of the zero-mean space
\begin{equation*}
\mathcal N=\left\{v_h\in V_h:\int_\Omega v_h\,dx=0\right\},
\end{equation*}
following the pairwise construction in~\cite{NNG:2017}.
Let $\mathcal E_j$ be the set of mesh edges on level $j$ and define
\begin{equation*}
\psi_{j,kl}=\phi_{j,k}-\frac{\int_\Omega\phi_{j,k}\,dx}{\int_\Omega\phi_{j,l}\,dx}\phi_{j,l},
\quad (k,l)\in\mathcal E_j.
\end{equation*}
Then we can verify that
\begin{equation}
\label{mean_decomposition}
\mathcal N=\sum_{j=1}^J\sum_{(k,l)\in\mathcal E_j}\operatorname{span}\{\psi_{j,kl}\}.
\end{equation}
These generating functions are not linearly independent in general.

On each level $j$, we compare the following two choices of local subspaces $V_{j,k}$:
\begin{itemize}
\item \texttt{Nodal}: We use all nodal subspaces $\operatorname{span}\{\phi_{j,k}\}$.
\item \texttt{Nodal + Pair}: We use the nodal subspaces together with $\operatorname{span}\{\psi_{j,kl}\}$, visiting the nodal subspaces first.
\end{itemize}
The \texttt{Nodal} method fails~\eqref{kernel_decomposition}, whereas \texttt{Nodal + Pair} satisfies it by~\eqref{mean_decomposition}.
Pair corrections are included on every level to treat mean-zero error at the corresponding mesh scales.
The restriction $\Pi_j$ is nodal interpolation followed by a constant correction that preserves the mean.
We use the $L^2(\Omega)$-norm in~\eqref{inner_stopping}.
For the outer-convergence experiments, we use the second method with $h=2^{-5}$, $r=3/2,2,3$, and $\epsilon=1,10^{-1},10^{-2}$.
\Cref{Fig:mean_outer} shows the $L^2$ error of the augmented Lagrangian iterates with respect to the discrete solution $u_h$.
The discrete energy is strongly convex and has a Lipschitz continuous gradient on bounded sets, while the mean constraint operator is surjective.
Thus, \cref{Prop:E_d,Rem:local_conditions} give the dual uniform convexity exponent $p=2$ on the relevant bounded sets.
For $r=2$, the nearly straight semilogarithmic curves agree with the linear estimate in \cref{Cor:primal}.
For $r=3>2$, the error decreases increasingly rapidly, consistently with the superlinear upper bound, whereas $r=3/2<2$ gives the slower, sublinear behavior.
Smaller $\epsilon$ accelerates the outer iteration in all three cases, as predicted by \cref{Cor:equiv}.

To examine the dependence of the subproblem solvers on $h$ and $\epsilon$, we compare the average numbers of FAS V-cycles over the $20$ augmented Lagrangian iterations for $h=2^{-5},2^{-6},2^{-7}$ and $r=3/2,2,3$.

\Cref{Tab:mean_subspace} shows comparable average counts for \texttt{Nodal + Pair} across the three meshes and the tested penalty range.
For $r=3/2$ and $r=2$, \texttt{Nodal} also benefits from coarse corrections when the penalty is moderate, but reaches the iteration limit for small $\epsilon$.
As discussed in \cref{Sec:Computation}, the penalty term becomes dominant as $\epsilon$ decreases but is constant along the kernel directions.
The multilevel mean-zero pair subspaces permit corrections in these directions without changing the penalty term.
The comparison illustrates that coarse resolution and the kernel decomposition address distinct sources of slow convergence.

% Subsection: Darcy--Forchheimer model
\subsection{Darcy--Forchheimer model}

\begin{figure}
\centering\includegraphics[width=\textwidth]{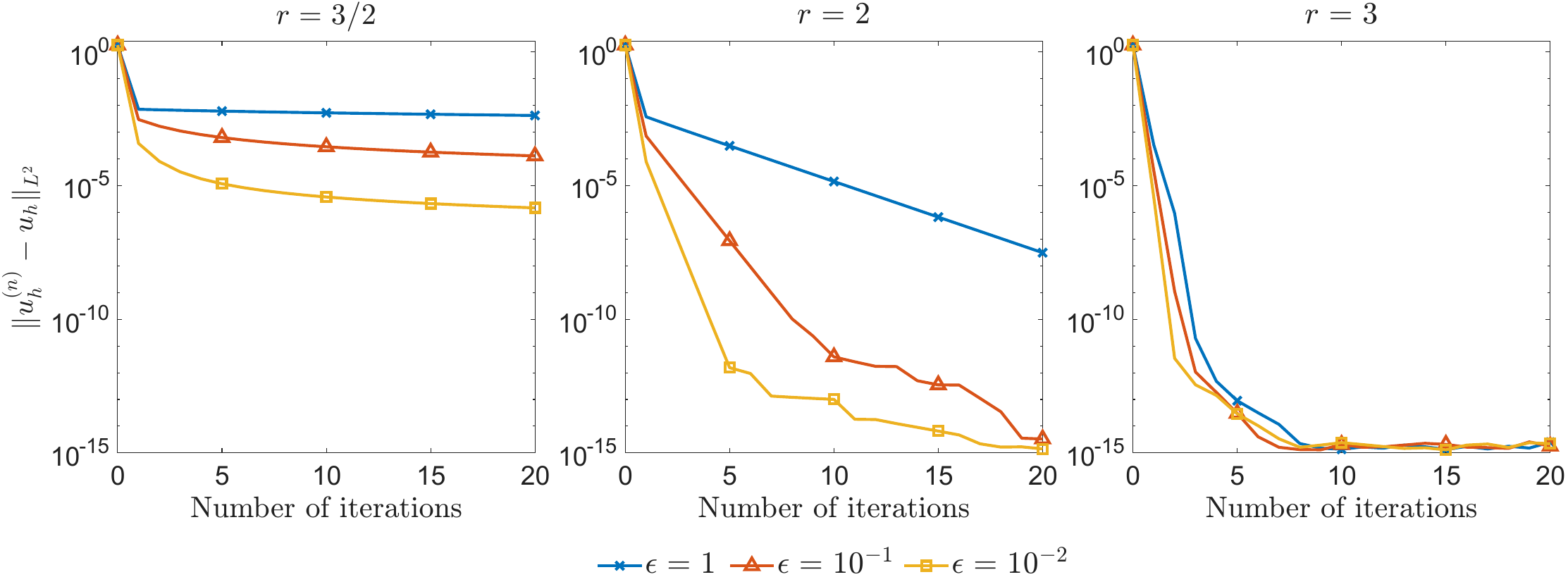}
\caption{Velocity $L^2$-norm error $\| \boldsymbol{u}_h^{(n)} - \boldsymbol{u}_h \|_{L^2 (\Omega)}$ of the high-order augmented Lagrangian method~\eqref{Forchheimer_ALM} for solving the Darcy--Forchheimer model~\eqref{Forchheimer_opt} with $h=2^{-5}$. Each augmented subproblem is solved by the multilevel \texttt{Edge + Curl} method.}
\label{Fig:Forchheimer_u}
\end{figure}
\begin{figure}
\centering\includegraphics[width=\textwidth]{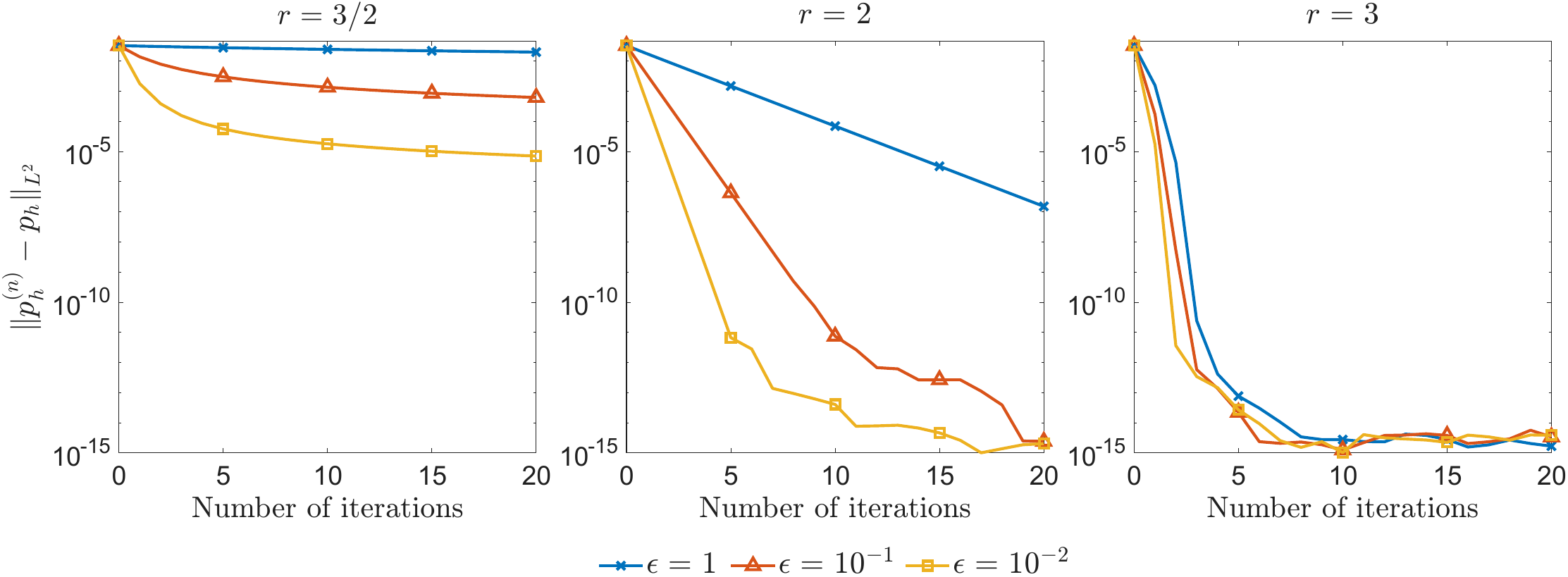}
\caption{Pressure $L^2$-norm error $\| p_h^{(n)} - p_h \|_{L^2 (\Omega)}$ of the high-order augmented Lagrangian method~\eqref{Forchheimer_ALM} for solving the Darcy--Forchheimer model~\eqref{Forchheimer_opt} with $h=2^{-5}$. Each augmented subproblem is solved by the multilevel \texttt{Edge + Curl} method.}
\label{Fig:Forchheimer_p}
\end{figure}
\begin{table}
\centering
\small
\caption{Average numbers of FAS V-cycles over $20$ iterations of the high-order augmented Lagrangian method~\eqref{Forchheimer_ALM} for solving the Darcy--Forchheimer model~\eqref{Forchheimer_opt}. A dash indicates that a subproblem reached the limit of $2\times10^2$ cycles.}
\label{Tab:Darcy_subspace}
\begin{tabular}{crlrrrrr}
\hline
$h$ & $r$ & Subspaces & $\epsilon=1$ & $10^{-2}$ & $10^{-4}$ & $10^{-6}$ & $10^{-8}$ \\
\hline
$2^{-5}$ & $3/2$ & \texttt{Edge} & 43.3 & -- & -- & -- & -- \\
 &  & \texttt{Edge + Curl} & 7.9 & 7.7 & 4.5 & 2.5 & 1.9 \\
 & 2 & \texttt{Edge} & -- & -- & -- & -- & -- \\
 &  & \texttt{Edge + Curl} & 8.3 & 3.1 & 2.2 & 1.9 & 1.9 \\
 & 3 & \texttt{Edge} & -- & -- & -- & -- & -- \\
 &  & \texttt{Edge + Curl} & 3.0 & 2.1 & 2.0 & 1.9 & 1.8 \\
\hline
$2^{-6}$ & $3/2$ & \texttt{Edge} & -- & -- & -- & -- & -- \\
 &  & \texttt{Edge + Curl} & 8.6 & 7.2 & 4.3 & 2.2 & 2.0 \\
 & 2 & \texttt{Edge} & -- & -- & -- & -- & -- \\
 &  & \texttt{Edge + Curl} & 8.0 & 3.3 & 2.3 & 2.1 & 1.8 \\
 & 3 & \texttt{Edge} & -- & -- & -- & -- & -- \\
 &  & \texttt{Edge + Curl} & 2.9 & 2.2 & 2.0 & 2.4 & 1.9 \\
\hline
$2^{-7}$ & $3/2$ & \texttt{Edge} & -- & -- & -- & -- & -- \\
 &  & \texttt{Edge + Curl} & 9.0 & 7.2 & 4.0 & 2.4 & 1.9 \\
 & 2 & \texttt{Edge} & -- & -- & -- & -- & -- \\
 &  & \texttt{Edge + Curl} & 8.0 & 2.9 & 2.2 & 2.0 & 1.9 \\
 & 3 & \texttt{Edge} & -- & -- & -- & -- & -- \\
 &  & \texttt{Edge + Curl} & 2.8 & 2.3 & 2.0 & 2.3 & 1.8 \\
\hline
\end{tabular}
\end{table}

Next, we consider the Darcy--Forchheimer model.
This model describes a nonlinear relationship between the flow rate and the pressure gradient in porous media~\cite{SMH:2023}.
Given a bounded domain $\Omega \subset \mathbb{R}^d$, the steady-state Darcy--Forchheimer model consists of three equations: a nonlinear relationship between the Darcy velocity $\boldsymbol{u}$ and the pressure $p$, conservation of mass, and a boundary condition:
\begin{equation}
\label{Forchheimer}
\begin{aligned}
\frac{\mu}{\rho} \boldsymbol{K}^{-1} \boldsymbol{u} + \frac{\beta}{\rho} | \boldsymbol{u} | \boldsymbol{u} + \nabla p = \boldsymbol{f} \quad &\text{ in } \Omega, \\
\operatorname{div} \boldsymbol{u} = g \quad &\text{ in } \Omega, \\
p = 0 \quad &\text{ on } \partial \Omega,
\end{aligned}
\end{equation}
where $\boldsymbol{K}$ is the permeability tensor, $\rho$ is the fluid density, $\mu$ is the fluid viscosity, $\beta$ is the Forchheimer coefficient, $\boldsymbol{f}$ is the external body force per unit volume, and $g$ is a prescribed source or sink term for mass.

The finite element formulation of~\eqref{Forchheimer} using a Raviart--Thomas-type conforming mixed finite element space~\cite{BBF:2013} $X_h \times M_h \subset X \times M$, where
\begin{equation*}
    X = \{ \boldsymbol{v} \in L^3 (\Omega)^d : \operatorname{div} \boldsymbol{v} \in L^2 (\Omega) \}, \quad 
    M = L^2 (\Omega),
\end{equation*}
is given as follows: find $(\boldsymbol{u}_h, p_h) \in X_h \times M_h$ such that
\begin{equation}
\label{Forchheimer_FEM}
\begin{aligned}
\frac{\mu}{\rho} \int_{\Omega} \boldsymbol{K}^{-1} \boldsymbol{u}_h \cdot \boldsymbol{v} \,dx
+ \frac{\beta}{\rho} \int_{\Omega} | \boldsymbol{u}_h | \boldsymbol{u}_h \cdot \boldsymbol{v} \,dx
- \int_{\Omega} p_h \operatorname{div} \boldsymbol{v} \,dx
&= \int_{\Omega} \boldsymbol{f} \cdot \boldsymbol{v} \,dx,
\quad \boldsymbol{v} \in X_h, \\
\int_{\Omega} q \operatorname{div} \boldsymbol{u}_h \,dx
&= \int_{\Omega} g q \,dx,
\quad q \in M_h.
\end{aligned}
\end{equation}
One may refer to~\cite{PR:2012} for further details.

In~\cite{PH:2025}, the following convex minimization problem with a divergence constraint, which is equivalent to~\eqref{Forchheimer_FEM}, is considered:
\begin{equation}
\label{Forchheimer_opt}
\resizebox{\textwidth}{!}{$\displaystyle
    \min_{\boldsymbol{v} \in X_h} \left\{
    F (\boldsymbol{v}) := \frac{\mu}{2 \rho} \int_{\Omega} \boldsymbol{v}\cdot\boldsymbol{K}^{-1}\boldsymbol{v} \,dx
    + \frac{\beta}{3 \rho} \int_{\Omega} | \boldsymbol{v} |^3 \,dx
    - \int_{\Omega} \boldsymbol{f} \cdot \boldsymbol{v} \,dx
    \right\}
    \text{ subject to } \operatorname{div} \boldsymbol{v} = g_h,
$}
\end{equation}
where $g_h$ is the $L^2 (\Omega)$-orthogonal projection of $g$ onto $M_h$.
The relevant kernel is given by
\begin{equation*}
\mathcal N=\{\boldsymbol v\in X_h:\operatorname{div}\boldsymbol v=0\}.
\end{equation*}
Here, we consider the high-order augmented Lagrangian method applied to the discrete Darcy--Forchheimer model~\eqref{Forchheimer_opt}.
We use the lowest-order Raviart--Thomas discretization, so that $M_h$ consists of piecewise constant functions.
We represent the multipliers in $M_h^*$ by pressure functions in $M_h$ using the $L^2(\Omega)$-inner product and equip $M_h^*$ with the $L^r(\Omega)$-norm of these representatives.
Then, by setting
\begin{equation*}
    B \leftarrow - \operatorname{div}, \quad
    g \leftarrow - g_h,
\end{equation*}
\cref{Alg:ALM} equipped with the stable dual update~\eqref{dual_update_stable} can be written as follows:
\begin{equation}
\label{Forchheimer_ALM}
\begin{aligned}
\boldsymbol{u}_h^{(n+1)} &=
\operatornamewithlimits{\arg\min}_{\boldsymbol{v} \in X_h}
\left\{
F (\boldsymbol{v})
- \int_{\Omega} p_h^{(n)} \operatorname{div} \boldsymbol{v} \,dx
+ \frac{1}{r^* \epsilon^{r^* - 1}} \int_{\Omega} | \operatorname{div} \boldsymbol{v} - g_h |^{r^*} \,dx
\right\}, \\
p_h^{(n+1)} &= (\operatorname{div} \operatorname{div}^*)^{-1} \operatorname{div} F' (\boldsymbol{u}_h^{(n+1)}) ,
\end{aligned}
\end{equation}
where $\operatorname{div}^*\colon M_h^*\to X_h^*$ is the adjoint of $\operatorname{div}\colon X_h\to M_h$, and the compositions use the $L^2(\Omega)$ Riesz representatives.
Note that~\eqref{Forchheimer_ALM} reduces to~\cite[Algorithm~3.1]{PH:2025} when $r = 2$.

In the numerical experiments, we use the setting from~\cite[Example~5.1]{PH:2025}.
Specifically, we set $\Omega = (0,1)^2 \subset \mathbb{R}^2$ and assume that $\boldsymbol{K} = K \boldsymbol{I}$ for some positive scalar $K$.
We take $\mu = 1$, $\rho = 1$, $K = 1$, and $\beta = 10$.
The functions $\boldsymbol{f}$ and $g$, together with the corresponding exact solutions $\boldsymbol{u}$ and $p$ of~\eqref{Forchheimer}, are given by
\begin{align*}
\boldsymbol{f} (x,y) &= \left(\frac{\mu}{\rho K} + \frac{\beta}{\rho} e^x \right)(e^x \sin y, e^x \cos y)^t + (y(1-2x)(1-y), x(1-x)(1 - 2y))^t, \\
g (x,y) &= 0, \quad
\boldsymbol{u} (x,y) = ( e^x \sin y, e^x \cos y)^t, \quad
p (x,y) = xy(1-x)(1-y).
\end{align*}
We use nested rectangular grids with mesh widths $h_j=2^{-j}$, $1\leq j\leq J$, and write $h=h_J$.
On each level $j$, the local subspaces $V_{j,k}$ are chosen in the following two ways:
\begin{itemize}
\item \texttt{Edge}: We use all Raviart--Thomas edge subspaces on every level.
\item \texttt{Edge + Curl}: We use the edge subspaces followed by the one-dimensional spaces generated by curls of bilinear nodal vertex functions on every level~\cite{AFW:2000}.
\end{itemize}
The first decomposition fails~\eqref{kernel_decomposition} because every edge basis function has nonzero divergence, whereas the second satisfies it because the vertex curls span the divergence-free kernel.
Raviart--Thomas prolongation preserves the divergence-free coarse corrections.
The restriction $\Pi_j$ averages normal fluxes on coarse edges, and the coarse energies use the FAS residual correction.
We use the finite element $L^2$ inner products for both velocity and pressure.
We use the $L^2(\Omega)$-norm in~\eqref{inner_stopping}.

For the outer-convergence experiments, we use the second method with $h=2^{-5}$, $r=3/2,2,3$, and $\epsilon=1,10^{-1},10^{-2}$.
In the Darcy--Forchheimer model, both the velocity and the pressure are important physical quantities.
Accordingly, we present numerical results for both variables, as shown in \cref{Fig:Forchheimer_u,Fig:Forchheimer_p}.
The errors are evaluated in their finite element $L^2$-norms.
We observe the same dependence of the outer convergence on $r$ and $\epsilon$ as in the case of mean-constrained nonlinear diffusion.

\Cref{Tab:Darcy_subspace} compares the average numbers of FAS V-cycles over the $20$ augmented Lagrangian iterations for $h=2^{-5},2^{-6},2^{-7}$ and $r=3/2,2,3$.
\texttt{Edge + Curl} maintains comparable average counts as the mesh is refined and $\epsilon$ decreases.
\texttt{Edge} reaches the iteration limit in most of these tests, including $\epsilon=1$ on the finer meshes.
The multilevel curl corrections treat divergence-free components on all mesh scales without changing the penalty term.

% Subsection: Finite neuron method
\subsection{Finite neuron method}

\begin{figure}
\centering\includegraphics[width=\textwidth]{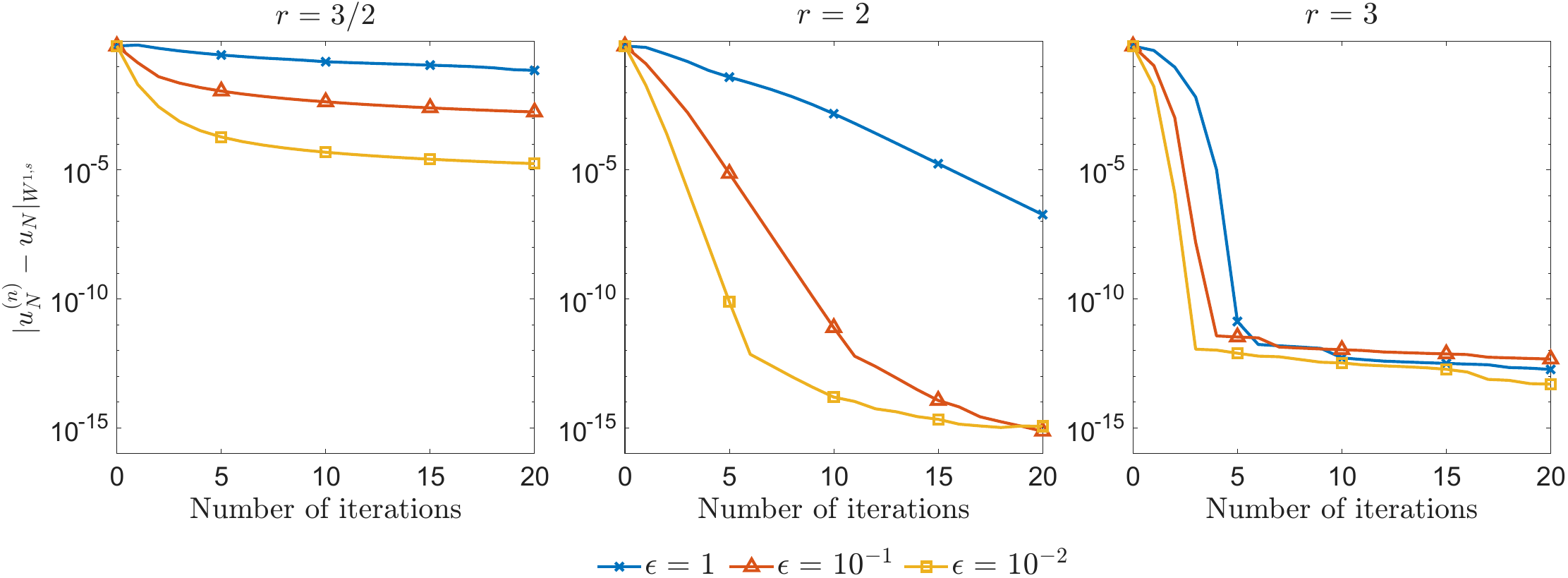}
\caption{$W^{1,s}$-seminorm error $| u_N^{(n)} - u_N |_{W^{1,s}(\Omega)}$ of the high-order augmented Lagrangian method~(\cref{Alg:ALM}) for solving the finite neuron discretization~\eqref{sLap_FNM} with $s=3$, $f(x)=1+x$, and $h=2^{-5}$. Each augmented subproblem is solved by the \texttt{Multilevel Hat} method.}
\label{Fig:FNM_3.0}
\end{figure}
\begin{figure}
\centering\includegraphics[width=\textwidth]{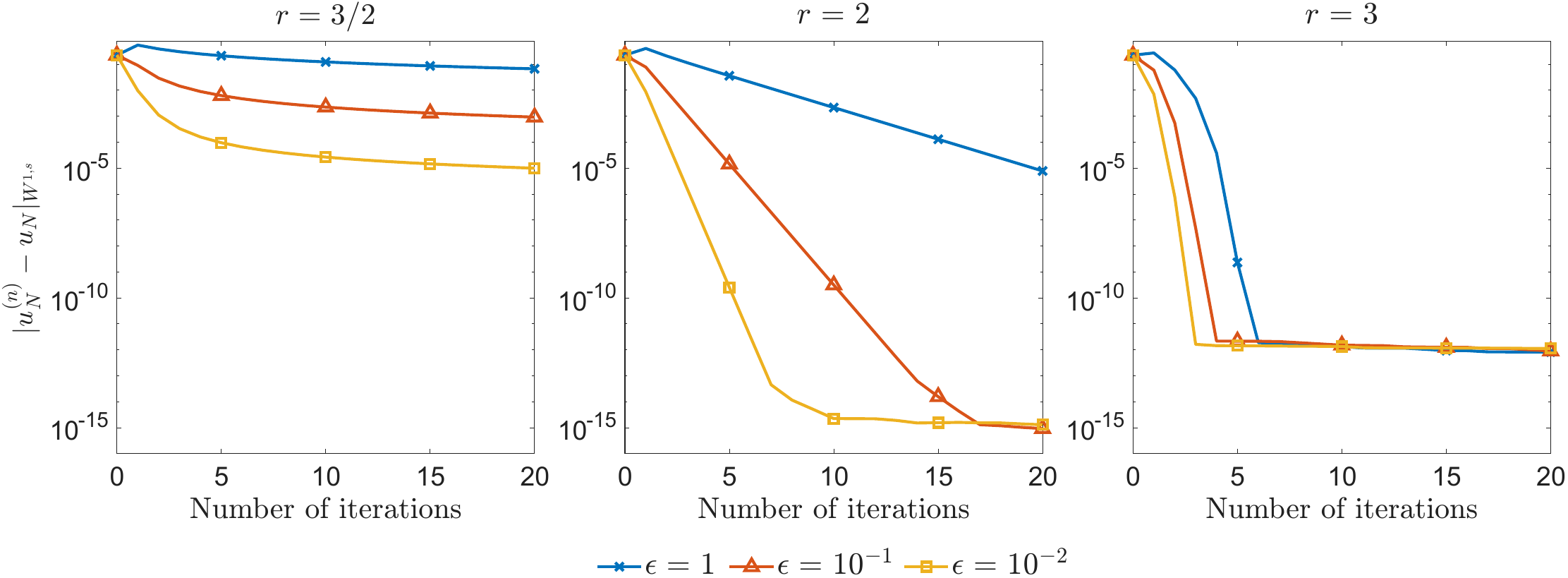}
\caption{$W^{1,s}$-seminorm error $| u_N^{(n)} - u_N |_{W^{1,s}(\Omega)}$ of the high-order augmented Lagrangian method~(\cref{Alg:ALM}) for solving the finite neuron discretization~\eqref{sLap_FNM} with $s=3/2$, $f(x)=1+x$, and $h=2^{-5}$. Each augmented subproblem is solved by the \texttt{Multilevel Hat} method.}
\label{Fig:FNM_1.5}
\end{figure}
\begin{table}
\centering
\small
\caption{Average numbers of inner iterations over $20$ iterations of the high-order augmented Lagrangian method~(\cref{Alg:ALM}) for solving the finite neuron problem~\eqref{sLap_FNM} with $s=3/2$ and $f(x)=1+x$. One inner iteration is an FAS V-cycle for \texttt{Multilevel Hat} and a fine-level sweep for \texttt{Neuron}. A dash indicates that a subproblem reached the limit of $2\times10^4$ iterations.}
\label{Tab:FNM_subspace_1.5}
\begin{tabular}{crlrrrrr}
\hline
$h$ & $r$ & Subspaces & $\epsilon=1$ & $10^{-2}$ & $10^{-4}$ & $10^{-6}$ & $10^{-8}$ \\
\hline
$2^{-5}$ & $3/2$ & \texttt{Multilevel Hat} & 22.4 & 13.4 & 7.7 & 3.0 & 2.4 \\
 &  & \texttt{Neuron} & 3244.5 & -- & -- & -- & -- \\
 & 2 & \texttt{Multilevel Hat} & 19.9 & 4.6 & 3.0 & 2.5 & 2.4 \\
 &  & \texttt{Neuron} & 2307.5 & -- & -- & -- & -- \\
 & 3 & \texttt{Multilevel Hat} & 5.5 & 3.3 & 2.7 & 3.9 & 4.3 \\
 &  & \texttt{Neuron} & -- & -- & -- & -- & -- \\
\hline
$2^{-6}$ & $3/2$ & \texttt{Multilevel Hat} & 22.1 & 13.8 & 8.1 & 3.3 & 2.5 \\
 &  & \texttt{Neuron} & 7176.2 & -- & -- & -- & -- \\
 & 2 & \texttt{Multilevel Hat} & 19.8 & 4.8 & 3.1 & 2.5 & 2.2 \\
 &  & \texttt{Neuron} & 4662.9 & -- & -- & -- & -- \\
 & 3 & \texttt{Multilevel Hat} & 5.6 & 3.4 & 2.6 & 3.4 & 3.8 \\
 &  & \texttt{Neuron} & -- & -- & -- & -- & -- \\
\hline
$2^{-7}$ & $3/2$ & \texttt{Multilevel Hat} & 21.9 & 13.2 & 7.5 & 3.1 & 2.4 \\
 &  & \texttt{Neuron} & -- & -- & -- & -- & -- \\
 & 2 & \texttt{Multilevel Hat} & 19.4 & 4.5 & 3.1 & 2.5 & 2.4 \\
 &  & \texttt{Neuron} & -- & -- & -- & -- & -- \\
 & 3 & \texttt{Multilevel Hat} & 5.5 & 3.2 & 2.7 & 3.0 & 3.1 \\
 &  & \texttt{Neuron} & -- & -- & -- & -- & -- \\
\hline
\end{tabular}
\end{table}

\begin{table}
\centering
\small
\caption{Average numbers of inner iterations over $20$ iterations of the high-order augmented Lagrangian method~(\cref{Alg:ALM}) for solving the finite neuron problem~\eqref{sLap_FNM} with $s=3$ and $f(x)=1+x$. One inner iteration is an FAS V-cycle for \texttt{Multilevel Hat} and a fine-level sweep for \texttt{Neuron}. A dash indicates that a subproblem reached the limit of $2\times10^4$ iterations.}
\label{Tab:FNM_subspace_3.0}
\begin{tabular}{crlrrrrr}
\hline
$h$ & $r$ & Subspaces & $\epsilon=1$ & $10^{-2}$ & $10^{-4}$ & $10^{-6}$ & $10^{-8}$ \\
\hline
$2^{-5}$ & $3/2$ & \texttt{Multilevel Hat} & 18.6 & 23.9 & 12.8 & 4.1 & 3.0 \\
 &  & \texttt{Neuron} & 2040.0 & -- & -- & -- & -- \\
 & 2 & \texttt{Multilevel Hat} & 20.6 & 5.3 & 3.9 & 3.2 & 3.0 \\
 &  & \texttt{Neuron} & 3564.8 & -- & -- & -- & -- \\
 & 3 & \texttt{Multilevel Hat} & 5.9 & 3.7 & 3.4 & 3.5 & 3.8 \\
 &  & \texttt{Neuron} & -- & -- & -- & -- & -- \\
\hline
$2^{-6}$ & $3/2$ & \texttt{Multilevel Hat} & 16.8 & 13.3 & 7.7 & 2.9 & 2.4 \\
 &  & \texttt{Neuron} & 3916.2 & -- & -- & -- & -- \\
 & 2 & \texttt{Multilevel Hat} & 14.9 & 4.2 & 2.9 & 2.4 & 2.2 \\
 &  & \texttt{Neuron} & 4358.1 & -- & -- & -- & -- \\
 & 3 & \texttt{Multilevel Hat} & 5.3 & 3.0 & 2.5 & 2.9 & 2.9 \\
 &  & \texttt{Neuron} & -- & -- & -- & -- & -- \\
\hline
$2^{-7}$ & $3/2$ & \texttt{Multilevel Hat} & 17.3 & 12.4 & 7.1 & 2.7 & 2.1 \\
 &  & \texttt{Neuron} & -- & -- & -- & -- & -- \\
 & 2 & \texttt{Multilevel Hat} & 16.5 & 4.0 & 2.6 & 2.2 & 2.1 \\
 &  & \texttt{Neuron} & -- & -- & -- & -- & -- \\
 & 3 & \texttt{Multilevel Hat} & 4.5 & 2.9 & 2.4 & 2.8 & 2.4 \\
 &  & \texttt{Neuron} & -- & -- & -- & -- & -- \\
\hline
\end{tabular}
\end{table}

As a final application, we consider the finite neuron method~\cite{Xu:2020}, also known as the deep Ritz method~\cite{EY:2018}, which is a scientific machine learning approach that employs neural networks to solve PDEs.
With the rapid growth and increasing importance of scientific machine learning, there have been numerous efforts to design efficient numerical solvers for the finite neuron method; see, for example,~\cite{PXX:2025,SHJHX:2023}.

As a model problem, we consider the following one-dimensional $s$-Laplacian equation ($s > 1$) on the interval $\Omega = (0,1) \subset \mathbb{R}$:
\begin{equation}
\label{sLap}
    - ( |u'|^{s-2} u')' = f \quad \text{ in } (0,1), \quad u(0) = u(1) = 0.
\end{equation}
In all numerical experiments for this problem, we set $f(x)=1+x$.
It is well known (see, e.g.,~\cite{LP:2025a}) that~\eqref{sLap} admits the following variational formulation:
\begin{equation}
\label{sLap_opt}
    \min_{v \in W_0^{1, s} (\Omega)} \left\{ F(v) := \frac{1}{s} \int_0^1 | v' |^s \,dx - \int_0^1 f v \,dx \right\}.
\end{equation}

In the finite neuron method~\cite{EY:2018,Xu:2020}, we discretize the variational formulation~\eqref{sLap_opt} using a neural network.
Here, we consider the following discrete space generated by a linearized ReLU shallow neural network~\cite{LMX:2025}:
\begin{equation}
\label{FNM_space}
    V_N = \operatorname{span} \left\{ \operatorname{ReLU} \left( x - \frac{i-1}{N} \right) : 1 \leq i \leq N \right\},
\end{equation}
where $\operatorname{ReLU}(x) = \max \{ x, 0 \}$ for $x \in \mathbb{R}$.
Note that $V_N$ satisfies the homogeneous Dirichlet boundary condition at $x = 0$ but not at $x = 1$.
By replacing the solution space of~\eqref{sLap_opt} with $V_N$ and enforcing the homogeneous Dirichlet boundary condition at $x = 1$ explicitly, we obtain the following finite neuron discretization of~\eqref{sLap_opt}:
\begin{equation}
\label{sLap_FNM}
    \min_{v \in V_N} F(v)
    \quad \text{subject to} \quad
    v(1) = 0.
\end{equation}
We denote the solution of~\eqref{sLap_FNM} by $u_N \in V_N$.

In the design of numerical solvers for the finite neuron method and related scientific machine learning approaches, such as physics-informed neural networks~\cite{RPK:2019}, enforcing Dirichlet boundary conditions explicitly is a nontrivial task and has been actively studied; see, for example,~\cite{SS:2022,WMIK:2023,YZ:2025}.
Here, we address this issue by solving~\eqref{sLap_FNM} using the high-order augmented Lagrangian method.

The space $V_N$ is also the space of continuous piecewise linear functions on the uniform mesh that vanish at $x=0$.
We use the nested meshes $h_j=2^{-j}$, $1\leq j\leq J$, with $h=h_J=1/N$.
Let $\phi_{j,k}$ be the interior nodal hat functions on level $j$.
For $Bv=v(1)$, we have
\begin{equation*}
\mathcal N=\sum_{j=1}^J\sum_{k=1}^{2^j-1}\operatorname{span}\{\phi_{j,k}\},
\quad
V_N=\mathcal N+\operatorname{span}\{x\}.
\end{equation*}
On each level $j$, we take $n_j=2^j$ and choose the local subspaces $V_{j,k}$ to be $\operatorname{span}\{\phi_{j,k}\}$ for $1\leq k<2^j$ and $\operatorname{span}\{x\}$ for $k=2^j$.
The \texttt{Multilevel Hat} solver uses these subspaces in an FAS V-cycle.
The restriction $\Pi_j$ is nodal interpolation, which preserves the endpoint constraint, and the FAS coarse load matches the restricted fine gradient.
This decomposition preserves the neural approximation space and satisfies~\eqref{kernel_decomposition}.

We use the inner product $(v,w)=\int_0^1v'w'\,dx$ on $V_N$.
We use the $W^{1,s}(\Omega)$-seminorm in~\eqref{inner_stopping}.
For the outer-convergence experiments, we use $N = 32$ neurons in~\eqref{FNM_space}, so that $h=2^{-5}$.
We use $r=3/2,2,3$ and $\epsilon=1,10^{-1},10^{-2}$.

The $W^{1,s}$-seminorm error $| u_N^{(n)} - u_N |_{W^{1,s} (\Omega)}$, which is an important measure for the $s$-Laplacian problem~\eqref{sLap}, is plotted over the iterations of the high-order augmented Lagrangian method and shown in \cref{Fig:FNM_3.0,Fig:FNM_1.5} for $s = 3$ and $s = 3/2$, respectively.
We observe behavior similar to that in the mean-constrained nonlinear diffusion problem~\eqref{mean_energy}.

The comparison method, denoted by \texttt{Neuron}, uses the original neuron subspaces in~\eqref{FNM_space}, as in neuron-wise subspace correction~\cite{PXX:2025}.
Each neuron has nonzero value at $x=1$, so its one-dimensional span has zero intersection with $\mathcal N$.
Consequently, this decomposition fails~\eqref{kernel_decomposition}.
\Cref{Tab:FNM_subspace_1.5,Tab:FNM_subspace_3.0} compare the average numbers of inner iterations over the $20$ augmented Lagrangian iterations for $N=32,64,128$ and $r=3/2,2,3$.
One iteration is an FAS V-cycle for \texttt{Multilevel Hat} and a fine-level sweep for \texttt{Neuron}.
The \texttt{Multilevel Hat} counts remain comparable across the three meshes and the tested penalty range, while \texttt{Neuron} requires substantially more sweeps or reaches the iteration limit.

% Section: Monotone inclusions
\section{Extension to uniformly monotone inclusions}
\label{Sec:Monotone}
In this section, we briefly discuss how to extend the proximal point argument to monotone inclusions without an objective functional.
Let $A\colon V\rightrightarrows V^*$ be maximally monotone, and suppose that
\begin{equation*}
0\in Au
\end{equation*}
has a solution.
The high-order proximal point iteration is given by
\begin{equation}
\label{monotone_PPA}
0\in Au^{(n+1)}+\epsilon J_r(u^{(n+1)}-u^{(n)}),
\quad n\geq0.
\end{equation}
Under the norm assumptions in \cref{Sec:Preliminaries}, each step is well defined.
For an inner-product norm and $r=2$, this is the classical proximal point method~\cite{Rockafellar:1976a}; accelerated variants are discussed in~\cite{Kim:2021}.
Our purpose is to record a short counterpart of the preceding estimates under uniform monotonicity.

Suppose that $A$ is $(p,\mu)$-uniformly monotone for some $p\geq2$ and $\mu>0$, meaning that
\begin{equation*}
(\xi-\eta,v-w)\geq\mu\|v-w\|^p,
\quad \xi\in Av,\ \eta\in Aw.
\end{equation*}
The solution $u$ is then unique.
For a differentiable $(p,\mu_F)$-uniformly convex functional $F$, the operator $A=F'$ satisfies this condition with $\mu=2\mu_F/p$, by~\eqref{Bregman_sym}.

The following result can be proved by adapting the arguments in the proofs of \cref{Thm:PPM_linear,Thm:PPM_superlinear,Thm:PPM_sublinear} to the uniformly monotone setting.
We omit the proof.

\begin{theorem}
\label{Thm:monotone_rates}
In~\eqref{monotone_PPA}, suppose that $A$ is $(p,\mu)$-uniformly monotone for some $p\geq2$ and $\mu>0$.
The sequence generated by~\eqref{monotone_PPA} converges to the unique solution $u$, and the following hold:
\begin{enumerate}[(a)]
\item If $r=p$, then we have
\begin{equation*}
\|u^{(n+1)}-u\|
\leq\left(\frac{\epsilon}{\mu}\right)^{\frac{1}{p-1}}
\left(\frac{\epsilon}{\mu+\epsilon}\right)^{\frac{n}{p-1}}
\|u^{(1)}-u^{(0)}\|,
\quad n\geq0.
\end{equation*}
\item If $r>p$, then we have
\begin{equation*}
\|u^{(n+1)}-u\|
\leq\left(\frac{\epsilon}{\mu}\right)^{\frac{1}{p-1}}
\|u^{(n)}-u\|^{\frac{r-1}{p-1}},
\quad n\geq0.
\end{equation*}
That is, the sequence converges to the solution superlinearly.
\item If $1<r<p$, then we have
\begin{equation*}
\|u^{(n+1)}-u^{(n)}\|
=O(n^{-\frac{1}{p-r}}).
\end{equation*}
\end{enumerate}
\end{theorem}

% For example, let $F\colon V\to\overline{\mathbb{R}}$ be a proper lower semicontinuous functional that is $(p,\mu_F)$-uniformly convex, with uniform convexity interpreted through subgradients.
% Then $A=\partial F$ is maximally monotone and $(p,2\mu_F/p)$-uniformly monotone, and~\eqref{monotone_PPA} is the high-order proximal point method for $F$.
% Thus, \cref{Thm:monotone_rates} also applies to nonsmooth convex minimization.
% If $F(u^{(0)})<\infty$, the energy estimates in \cref{Sec:PPM} remain valid by the same proofs, using the subgradient inequality in place of differentiability.

% Section: Conclusion
\section{Conclusion}
\label{Sec:Conclusion}
In this paper, we studied high-order augmented Lagrangian methods with parameter-robust multilevel inner solvers for solving convex optimization problems with linear constraints.
We considered how multilevel subspace correction methods satisfying the kernel decomposition condition~\eqref{kernel_decomposition} can solve the nearly semicoercive subproblems robustly with respect to the penalty parameter.
In particular, we discussed how to construct mesh- and parameter-robust multilevel methods for nonlinear diffusion, flow in porous media, and scientific machine learning.
Moreover, we provided direct energy error estimates within the established theory of power proximal point and augmented Lagrangian methods~\cite{OBLP:2026}, together with a short extension to uniformly monotone inclusions.

To conclude the paper, we discuss two possible directions for future research, one theoretical and one application oriented.
On the theoretical side, the numerical results presented in \cref{Sec:Applications} indicate that the high-order augmented Lagrangian method sometimes exhibits better performance than predicted by our theoretical analysis.
More precisely, for the finite neuron discretization of the $s$-Laplacian with $s=3/2$, the method demonstrates linear convergence when the order $r$ is set to $2$, whereas the estimate based on the dual uniform convexity exponent $p=s^*=3$ gives only a sublinear bound.
These numerical observations suggest that a sharper local analysis may explain this improved behavior.

We note that a similar phenomenon was recently established in~\cite{LP:2025a}, where linear convergence of domain decomposition methods for solving the $s$-Laplacian was proved, despite the previously known convergence rate being only sublinear~\cite{TX:2002}.
We expect that a similar type of analysis may be developed to explain the enhanced convergence behavior of the method observed in our numerical experiments.

The second direction, on the application side, is the design and analysis of para\-meter-robust multilevel methods for other nearly semicoercive convex optimization problems arising in the sciences.
As discussed in \cref{Sec:Computation}, efficient solution of the primal subproblems is essential for large-scale applications of high-order augmented Lagrangian methods.
Developing suitable solvers for a wider range of constraints would broaden their applicability.

\bibliographystyle{siamplain}
\bibliography{refs_augmented_Lagrangian}

\end{document}